\documentclass[12pt,a4paper,reqno,oneside]{amsart}
\usepackage[utf8]{inputenc}
\usepackage{times}
\usepackage{amssymb}
\usepackage[mathscr]{euscript}
\usepackage{pdfpages}
\usepackage{hyperref}

\newtheorem{thm}{Theorem}[section]
\newtheorem{prop}[thm]{Proposition}
\newtheorem{lem}[thm]{Lemma}
\newtheorem{cor}[thm]{Corollary}

\theoremstyle{remark}
\newtheorem{rem}[thm]{Remark}
\newtheorem{ex}[thm]{Example}

\theoremstyle{definition}
\newtheorem{defn}[thm]{Definition}

\renewcommand{\phi}{\varphi}
\newcommand{\A }{\mathcal A}
\newcommand{\B}{\mathcal B}

\newcommand{\U}{\mathcal U}

\renewcommand{\Re}{\mathrm{Re}}
\newcommand{\Tr}{\mathrm{Tr}}

\renewcommand{\>}{\rangle}

\begin{document}
\author[Kallsen]{Jan Kallsen}
\address[Jan Kallsen]{\\
Department of Mathematics\\
Kiel University \\
Westring 383\\
24853 Kiel\\ Germany}
\email[]{kallsen\@@math.uni-kiel.de}
\author[Kr\"uhner]{Paul Kr\"uhner}
\address[Paul Kr\"uhner]{\\
Financial \& Actuarial Mathematics\\
Vienna University of Technology\\
Wiedner Hauptstr. 8/E105-1\\
1040 Vienna\\
Austria}
\email[]{paulkrue\@@fam.tuwien.ac.at}

\title[On unique solutions to martingale problems]{On uniqueness of solutions to martingale problems --- \\ counterexamples and sufficient criteria}

\begin{abstract}
 The dynamics of a Markov process are often specified by its infinitesimal generator or, equivalently, its symbol. This paper contains examples of analytic symbols which do not determine the law of the corresponding Markov process uniquely. These examples also show that the law of a polynomial process in the sense of \cite{cuchiero.11, cuchiero.al.12} is not necessarily determined by its generator.  On the other hand, we show that a combination of smoothness of the symbol and ellipticity warrants uniqueness in law. The proof of this result is based on proving stability of univariate marginals relative to some properly chosen distance.\\


\end{abstract}
\maketitle

\section{Introduction}
Consider a system whose state at time $t$ is represented by a vector $X(t)$ in $\mathbb R^d$. In applications the dynamics of such a system are often described by specifying how $X(t)$ changes as a function of the current state $X(t)$. In a deterministic setup this is typically expressed in terms of an ordinary differential equation. If, on the other hand, $X(t)$ is random, it may be viewed as a Markov process whose local dynamics can be specified in terms of a stochastic differential equation, its infinitesimal generator, its local semimartingale characteristics, or its symbol. As in the deterministic case, this immediately leads to the question of existence and uniqueness of a stochastic process exhibiting the given local dynamics.

This can be rephrased in terms of existence and uniqueness of the solution to a corresponding martingale problem.
Existence is known to hold under relatively weak continuity conditions, cf.\ e.g.\ \cite[Theorem 4.5.4]{ethier.kurtz.86}, \cite[Theorem 3.15]{hoh.98a}, \cite[Theorem 3.24]{boettcher.al.13},
and Theorem \ref{Satz: Existenz unter Stetigkeit} below.
For continuous processes uniqueness holds for Lipschitz- resp.\ Hölder-continuous coefficients or under some ellipticity condition, cf.\ e.g.\ \cite{revuz.yor.99}. The situation is less obvious for processes with jumps. Lipschitz conditions only help for generators which have a natural representation as an SDE, which often is not the case. Ellipticity, on the other hand, requires a continuous martingale part to be present, which often is not the case either. 

This piece of research is motivated by the desire to come up with a general uniqueness result for Markov processes that may not have a continuous martingale part or a natural representation as a SDE. In this context we share the point of view of \cite{hoh.98a,jacob.schilling.01,boettcher.al.13} that it is natural to study Markov processes through their symbol.
Indeed, e.g.\ weak convergence of a sequence of Levy processes corresponds to pointwise convergence of their symbols.

From the analogy to ODE's one may expect uniqueness to hold if the symbol of the process depends smoothly on the state $X(t)$. Unfortunately, smoothness alone does not seem to suffice in order to warrant uniqueness. In Section \ref{s:examples} we present two examples of even analytic symbols where uniqueness in law of the corresponding Markov process does not hold. This is the first main result of this paper. These examples also show that the law of a polynomial process in the sense of \cite{cuchiero.11, cuchiero.al.12} is not in general uniquely determined by its generator. 

Section \ref{s:exuni} contains a positive result, which is the second main contribution of this paper. It is shown that the combination of sufficient smoothness and ellipticity warrants uniqueness in law. In contrast to \cite[Theorem 4.3]{stroock.75} and related results, the continuous martingale part may vanish. The probably closest relative to our Theorem \ref{Satz: Eindeutigkeit unter Elliptizitaet neu} below is  \cite[Theorem 2.8]{boettcher.05} which also relies on smoothness and ellipticity of the symbol. However, Böttcher requires a certain boundedness for derivatives of any order while we need this condition only for finitely many derivatives. Nevertheless, \cite[Theorem 2.8]{boettcher.05} is not a special case of our Theorem \ref{Satz: Eindeutigkeit unter Elliptizitaet neu} below.
Another closely related result is \cite[Theorem 5.24]{hoh.98a} which, however, requires the symbol to be real.

Uniqueness results have been obtained by a number of different approaches, cf.\ \cite{jacob.schilling.01} for an overview. From a very rough perspective, the most commonly used techniques are
\begin{itemize}
 \item SDE methods where uniqueness is often obtained from fixed-point arguments,
 \item construction of a solution to the backward equation, i.e.\ construction of solutions for the associated abstract Cauchy problem, and
 \item so-called interlacing techniques which allow to add finitely many jumps. 
\end{itemize}
By contrast, our approach is based on establishing stability of the univariate marginals relative to a properly chosen distance. This kind of reasoning seems to be new and it constitutes the third main contribution of this paper.

The paper is structured as follows. In Section \ref{s:symbol} we recall various notions
and properties concerning symbols and martingale problems. Moreover,
we state an existence result which follows from  \cite[Theorem 4.5.4]{ethier.kurtz.86}.
Subsequently, we present examples showing that smoothness of the symbol does not imply uniqueness
of the solution to a martingale problem.
In Section \ref{s:exuni} a uniqueness result under smoothness and some mild ellipticity of the
symbol is stated. Section \ref{proofs} contains proofs.
In the appendix, we recall some facts on complex measures.

\subsection{Notation}
$d\in\mathbb N$ generally denotes the dimension of the space under consideration.
We denote the trace of a matrix $C\in\mathbb R^{d\times d}$ by $\mathrm{Tr}(C)$. For any two vectors $x,y\in\mathbb C^d$ we define the standard bilinear form $xy:=\sum_{j=1}^dx_jy_j$. Moreover, we set $Cx := (\sum_{k=1}^dC_{jk}x_k)_{j=1,\dots,d}$ and $yCx := y(Cx)$ for any matrix $C\in\mathbb C^{d\times d}$ and vectors $x,y\in \mathbb C^d$. The set of positive semidefinite $d\times d$-matrices is denoted by $S^d$. We fix a {\em truncation function} $\chi:\mathbb R^d\rightarrow\mathbb R^d$, i.e.\ $\chi$ is measurable, bounded and it equals the identity in a neighbourhood of zero. W.l.o.g.\ we suppose that $\chi(x)=x$ for $|x|\leq1$, where
$\vert x\vert:=(\sum_{j=1}^d\vert x_j\vert^2)^{1/2}$ denotes the Euclidean norm on $\mathbb R^d$. We write $B(x,r) := \{y\in\mathbb R^d: \vert x-y\vert< r\}$ for the open ball with radius $r>0$ centered at $x\in\mathbb R^d$. We denote the gradient of $f\in C^1(\mathbb R^d,\mathbb C)$ by $\nabla f(x) := (\partial_1f,\dots,\partial_df)$,  $x\in\mathbb R^d$,
the Hessian of $f\in C^2(\mathbb R^d,\mathbb C)$ by $Hf(x):=(\partial_{jk}^2f(x))_{j,k=1}^d$, $x\in\mathbb R^d$, and  the Laplacian of $f\in C^2(\mathbb R^d\times\mathbb R^d,\mathbb C)$ by $\Delta f(x) := \sum_{j=1}^d\partial_j^2f(x)$,  $x\in\mathbb R^d$.  For functions $f:\mathbb R^d\times\mathbb R^d\rightarrow\mathbb C$ we write $\nabla_1f(x,y):=(\nabla f(\cdot,y))(x)$, $x,y\in\mathbb R^d$ if $f$ is differentiable with respect to the first coordinate  and $\Delta_1f(x,y):=(\Delta f(\cdot,y))(x)$, $x,y\in\mathbb R^d$ for sufficiently smooth $f$. 
If $f$ is smooth enough in the second coordinate,$\nabla_2f(x,y)$ and $H_2f(x,y)$ are defined accordingly.
By $\hat C(\mathbb R^d)$ (resp.\ $\hat C(\mathbb R^d,\mathbb C)$) we denote the set of real-valued (resp.\ complex-valued)
continuous functions on $\mathbb R^d$ that vanish in $\infty$.
The greatest integer less or equal $x\in\mathbb R$ is written as $[x]$.
Further unexplained notation is used as in \cite{ethier.kurtz.86, js.87}.

\section{The symbol and the existence theorem}\label{s:symbol}
We start by recalling the notion of the symbol and its associated martingale problem, cf.\  \cite{jacob.schilling.01,boettcher.al.13}.
A systematic theory for symbols was first developed by Hoh \cite{hoh.95,hoh.98,hoh.00}.
Other important references include \cite{jacob.93}, which is more in view of strongly continuous semigroups,
and \cite{bogachev.et.al.99}, who developed a theory for symbols on nuclear separable spaces.

\begin{defn}\label{Definition: Symbol}
\begin{enumerate}
 \item A measurable function $q:\mathbb R^d\times\mathbb R^d\rightarrow\mathbb C$ is a {\em symbol} if $q(x,\cdot)$ is a {\em L\'evy exponent} for all $x\in\mathbb R^d$, i.e.\ there are functions
$b:\mathbb R^d\rightarrow \mathbb R^d$,
$c:\mathbb R^d\rightarrow S^d$, and
$F:\mathbb R^d\times\mathcal B(\mathbb R^d) \rightarrow\mathbb R^d$
such that $F(x,\cdot)$ is a L\'evy measure and
\begin{equation}\label{e:symbol}
q(x,u) = iub(x)-\frac{1}{2}uc(x)u+\int_{}\left( e^{iuy}-1-iu\chi(y) \right)F(x,dy)
\end{equation}
for any $x,u\in\mathbb R^d$. 
We call a symbol $q:\mathbb R^d\times\mathbb R^d\rightarrow\mathbb C$ {\em ($f$-)Hölder continuous} if there is a continuous bounded function $f:\mathbb R^d\rightarrow\mathbb R_+$ such that
$$\vert q(x,u)-q(y,u)\vert\leq f(x-y)(1+\vert u\vert^2),\quad x,y,u\in\mathbb R^d.$$
 \item
If $q$ denotes a symbol,
an adapted c\`adl\`ag $\mathbb R^d$-valued stochastic process $X$ is called {\em solution to the $q$-martingale problem} if the process
$$M_u(t):=\exp(iuX(t))-\int_0^t q(X(s),u)\exp(iuX(s))ds$$
is a local martingale for any $u\in\mathbb R^d$. {\em Uniqueness} for the $q$-martingale problem means that any two solutions $X,Y$ to the $q$-martingale problem with the same initial law (i.e.\ $X(0)$ has the same law as $Y(0)$) have the same distribution. Finally, we say that {\em existence} holds for the $q$-martingale problem if, for any probability measure $\mu$, there is a solution $X$ to the $q$-martingale problem with initial law $P^{X(0)}=\mu$.
\end{enumerate}
\end{defn}
\begin{rem}\label{Bemerkung: triplet}
The functions $b,c$ in  the L\'evy-Khintchine representation (\ref{e:symbol})
are measurable and $F$ is a transition kernel from $\mathbb R^d$ to $\mathbb R^d$,
which can e.g.\ be derived from the construction in \cite[Lemma II.2.44]{js.87}.
As opposed to $c$ and $F$, the drift coefficient $b$ depends on the choice of $\chi$, cf.\ \cite[Theorem 8.1]{sato.99}. 
We call $(b(x),c(x),F(x,\cdot))$  {\em L\'evy-Khintchine triplet} of the L\'evy exponent $q(x,\cdot)$.
\end{rem}

For further use we make the following observation.
\begin{rem}
 Let $q$ be a symbol. A simple Taylor approximation argument shows that 
 for any $x\in\mathbb R^d$ there exists $C_x<\infty$ such that  $\vert q(x,u)\vert \leq C_x(1+\vert u\vert^2)$,
 $u\in\mathbb R^d$. Hence 
 $$ \int_{} \vert f(u)q(x,u)\vert du < \infty $$
 for any Schwartz function $f$  in the sense of 
\cite[Definition 2.2.1]{grafakos.08}
and any $x\in \mathbb R^d$.
\end{rem}
In order to relate a symbol to a martingale problem in the sense of 
\cite[Section 4.3]{ethier.kurtz.86}, we define an operator corresponding to the symbol. 
\begin{defn}\label{Definition: Generator}
 Let $q$ be a symbol. The {\em operator $\A$ associated with $q$} is defined as
$$ \A  f(x) := \int_{}q(x,u)\check{f}(u)e^{iux}du$$
where $x\in\mathbb R^d$, $f$ is any real-valued Schwartz function,
and 
$$\check{f}:\mathbb R^d\rightarrow\mathbb C,\quad u\mapsto\frac{1}{(2\pi)^d}\int_{}f(x)e^{-iux}dx,$$ 
denotes the inverse Fourier transform of $f$.
\end{defn}
This operator can be expressed in terms of the L\'evy-Khintchine triplet.
\begin{lem}\label{l:Trippel Darstellung}
 Let $q$ be a symbol, $\A$ the operator associated with $q$, and $(b,c,F)$ the triplet of $q$. 
 Then
 $$\A f(x) = \nabla f(x)b(x) + \frac{1}{2}\mathrm{Tr}(Hf(x)c(x)) + \int_{} \left(f(x+y)-f(x)-\nabla f(x)\chi(y)\right)F(x,dy)$$
 for any real-valued Schwartz function $f$ and any $x\in\mathbb R^d$.
\end{lem}
\begin{proof}
 This follows from \cite[Proposition 2.3.22]{grafakos.08}.
\end{proof}

 Let $q$ be a symbol with associated operator $\A $. Moreover, denote by $\mathcal B$ the restriction of $\A $ to the set of real-valued Schwartz functions $f$ such that $\A f$ is bounded. The following lemma shows that any solution $X$ to the $q$-martingale problem is a solution to the martingale problem in the sense of \cite[Section 4.3]{ethier.kurtz.86} for $\mathcal B$. Under suitable conditions the converse is also true, cf.\ Theorem \ref{Satz: Existenz unter Stetigkeit}(2) below.
\begin{lem}\label{l:paulneu2}
  Let $X$ be a solution to the $q$-martingale problem,
  $f$ a real-valued Schwartz function such that $\mathcal Af$ is bounded, and
   $$ M_f(t) := f(X(t)) - \int_0^t \mathcal Af(X(s)) ds $$
  for any $t\geq 0$. Then $M_f$ is a martingale.
\end{lem}
\begin{proof}
  Theorem \ref{Satz: Semimartingaleigenschaft} states that $X$ is a semimartingale with local characteristics 
  $(b(X_-),$ $c(X_-),F(X_-,\cdot))$. 
  Thus It\=o's formula for the local characteristics \cite[Proposition 2.5]{kallsen.04} together with Lemma \ref{l:Trippel Darstellung} yield that
    \begin{align*}
        G(t,A) &:= \int1_A(f(X_{t-}+x)-f(X_{t-})) F_t(dx),\quad A\in\B\mbox{ with }0\notin A, \\
        \gamma(t) &:=  \nabla f(X(t-)) c(X(t-)) \nabla f(X(t-)), \\
        \beta(t) &:= \mathcal A f(X(t-)) + \int (h(y) -y) G(t,dy),
    \end{align*}
    is a version of the local characteristics of $f(X)$ relative to the truncation function $h$. 
    By \cite[Theorem II.2.42]{js.87}
     $$ f(X(t)) - \int_0^t \mathcal Af(X(s))ds $$
    is a local martingale. However, $M_f$ is bounded on compact time intervals and hence it is a martingale.
\end{proof}

We now state an essentially  well-known existence result which follows from \cite[Theorem 4.5.4]{ethier.kurtz.86},
cf.\ also \cite[Theorem 3.15]{hoh.98a} or  \cite[Theorem 3.24]{boettcher.al.13}. 
\begin{thm}[Existence]\label{Satz: Existenz unter Stetigkeit}
 Let $q$ be a continuous symbol with associated operator $\A $
 and triplet $(b,c,F)$.
 Assume that 
$$g:\mathbb R^d\rightarrow\mathbb R\cup\{\infty\},\quad x\mapsto\vert b(x)\vert +\Tr(c(x)) +\int_{}\vert y\vert^2 F(x,dy)$$
is bounded by some finite constant.
Then the following statements hold.
\begin{enumerate}
 \item For any probability measure $\mu$ on $\mathbb R^d$ there is a solution $X$ to the $q$-martingale problem with $P^{X(0)}=\mu$. 
 \item A stochastic process $X$ is a solution to the $q$-martingale problem if and only if
 $$ M_f(t) := f(X(t)) - \int_0^t\A f(X(s))ds,\quad t\in\mathbb R_+$$
 defines a martingale for any real-valued Schwartz function $f$ (or, equivalently, any
 smooth function with compact support),
 i.e.\ if and only if $X$ is a solution to the martingale problem $\A$ in the sense of
 \cite[Section 4.3]{ethier.kurtz.86}.
 \item
The operator $\A $ has the following properties:
\begin{enumerate}
 \item its range is contained in $\hat C(\mathbb R^d)$,
 \item it satisfies the positive maximum principle in the sense of \cite[p.~165]{ethier.kurtz.86}, 
 i.e.\ 
 $$0\leq f(x_0)=\sup_{x\in\mathbb R^d}f(x)$$ implies $\A f(x_0)\leq0$ 
 for any real-valued Schwartz function $f$ and any $x_0\in\mathbb R^d$, and
 \item it is conservative, i.e.\ there is a bounded sequence of real-valued Schwartz functions $(f_n)_{n\in\mathbb N}$ which converges pointwise to $1$ such that $(\A f_n)_{n\in\mathbb N}$ is a bounded sequence which converges pointwise to $0$.
\end{enumerate}
\item
It is possible to choose measures $(P_x)_{x\in\mathbb R^d}$ on the Skorokhod space such that the canonical process $X$ is a solution to the $q$-martingale problem with $X(0) = x$ a.s.\ under $P_x$, $x\in\mathbb R^d$ and such that $x\mapsto P_x(X(t)\in A)$ is measurable for any $t\geq 0$ and any Borel set $A\subset\mathbb R^d$. Moreover, $(X,(P_x)_{x\in\mathbb R^d})$ is strong Markov. 
\item If the $q$-martingale problem has several solutions for some initial law $P^{X(0)}=\mu$, then there are several families of measures $(P_x)_{x\in\mathbb R^d}$ as in (4).
\item
Finally,
 $$ q(x,u) = \lim_{t\downarrow0} \frac{E_x(e^{iu(X(t)-x)})-1}{t},\quad x,u\in\mathbb R^d, $$
\end{enumerate}
where $E_x$ denotes expectation relative to $P_x$.
\end{thm}
The proof is to be found in Section \ref{Abschnitt: Bew. Existenzsatz}.

\begin{rem}
\begin{enumerate}
 \item Provided that uniqueness holds, the boundedness condition in Theorem \ref{Satz: Existenz unter Stetigkeit} can be relaxed by localisation, i.e.\ by applying \cite[Theorems 4.6.2 or 4.6.3]{ethier.kurtz.86}.
 \item For his related existence result \cite[Theorem 3.15]{hoh.98a} 
 (cf.\ also \cite[Theorem 3.24]{boettcher.al.13}),
 Hoh requires slightly weaker conditions. However, he focuses on statement (1),
 i.e., the existence of a Markovian solution is not considered.
 \item Statement (6) means that $-q$ is a symbol in the sense of \cite{jacob.schilling.01}.
\end{enumerate}
\end{rem}

The assumption on the triplet in Theorem \ref{Satz: Existenz unter Stetigkeit} can be replaced 
by a smoothness condition on the symbol:
\begin{cor}\label{co:exuntermomenten}
 Let $q$ be a continuous symbol such that $u\mapsto q(x,u)$ is twice differentiable
 with bounded gradient $x\mapsto \nabla_2q(x,0)$ and bounded Hessian $x\mapsto H_2q(x,0)$.
Then statements (1-6) in Theorem \ref{Satz: Existenz unter Stetigkeit} hold.
\end{cor}
\begin{proof}
  W.l.o.g.\ $|\chi(y)|=0$ for $|y|>1$.
  Fix $x\in\mathbb R^d$ and define the finite measure 
 $F_x(A):=F(x,A\setminus \overline{B(0,1)})$, $A\in\mathcal B(\mathbb R^d)$
  as well as $\overline F_x:=F(x,\cdot)-F_x$.
 We denote the Fourier transform of $F_x$ by
  $$ \hat F_x(u) := \int e^{iuy} F_x(dy),\quad u\in\mathbb R^d. $$
 Observe that
$$   \hat F_x(u)
= q(x,u)-iub(x)+{1\over2}uc(x)u-\int\big(e^ {iuy}-1-iuy\big)\overline F_x(dy)+\hat F_x(0).$$
 Dominated convergence and $\int |y|^2  \overline F_x(dy)<\infty $ yield that 
$\hat F_x$ is twice differentiable in $0$.
 By \cite[Lemma A.1]{duffie.al.03}, this implies that  $F_x$ has finite second moments given by
 $\int y_j^2  F_x(dy) =-\hat F_x''(0)$, $j=1,\dots,d$. 
 Again by dominated convergence we obtain
\begin{align*}
      (H_2q(x,0))_{jj} &= - c(x)_{jj} - \int y_j^2 \overline F_x(dy)- \int y_j^2 F_x(dy)\\
& =- c(x)_{jj} - \int y_j^2 F(x,dy),\quad j=1,\dots,d.
\end{align*}
   Boundedness of $H_2q(\cdot,0)$ now yields that $\mathrm{Tr}(c(\cdot))$ and $\int \vert y\vert^2 F(\cdot,dy)$ are bounded as well.
 
   Once more from dominated convergence we conclude
    \begin{align*}
      \nabla_2q(x,0) = i\left( b(x) + \int (y-\chi(y)) F(x,dy) \right).
    \end{align*}
   Since  $\int (y-\chi(y)) F(\cdot,dy)$ is a bounded function, $b$ is bounded as well. 
   Theorem \ref{Satz: Existenz unter Stetigkeit} now yields the assertion.
\end{proof}

Next we restate a result by van Casteren which shows that uniqueness implies the Feller property.
\begin{prop}[Feller property]\label{p:feller}
  Let $q$ be a symbol satisfying the requirements of Theorem \ref{Satz: Existenz unter Stetigkeit}. If uniqueness holds for the $q$-martingale problem, there is a closed extension $\mathcal C$ of $\A $ which generates a strongly continuous positivity preserving contraction semigroup on $\hat C(\mathbb R^d)$. In other words, any solution to the martingale problem $\A $ is a Feller process.
\end{prop}
\begin{proof}
Boundedness of $g$ in Theorem  \ref{Satz: Existenz unter Stetigkeit}
implies that $\A$ maps Schwartz functions on a subset of $\hat C(\mathbb R^d)$.
In view of \cite[Theorem 3.1]{casteren.okitaloshima.96} or \cite[Proposition 5.18]{hoh.98a}, this yields the claim.
\end{proof}

\section{Counterexamples}\label{s:examples}
In this section we provide an example of a real-valued analytic symbol which fails to have the uniqueness property in the sense of Definition \ref{Definition: Symbol}. Moreover, we present a closely related example. Both correspond to polynomial processes in the sense of
\cite{cuchiero.11, cuchiero.al.12}, 
i.e.\ the extended operator $\A $ in the sense of \cite[Definition 2.3]{cuchiero.al.12} 
maps polynomials on polynomials of at most the same degree.
\begin{ex}\label{Beispiel: analytisches Symbol ohne Eindeutigkeit}
There is an analytic symbol, namely 

\begin{equation}\label{e:ex1}
q:\mathbb R\times\mathbb R\rightarrow\mathbb R,\quad (x,u)\mapsto
\begin{cases}
\frac{\cos(xu)-1}{x^2}&\text{for }x\neq0,\\
-\frac{u^2}{2}&\text{otherwise},
\end{cases}
\end{equation}
 satisfying the requirements of Theorem \ref{Satz: Existenz unter Stetigkeit} and having an entire extension to $\mathbb C\times\mathbb C$ where, however, uniqueness does not hold for the $q$-martingale problem. 

Moreover, there are solutions $X,Y$ to the $q$-martingale problem with $X(0)=0=Y(0)$ and
$P^{X(t)}\not=P^{Y(t)}$ for any $t>0$.
More generally, there are strong Markov processes $X$, $Y$ on $\mathbb R^d$ with the above symbol
which do not have the same law.
Moreover, $X$, $Y$ are {\em polynomial processes} in the sense of 
\cite[Definition 2.1]{cuchiero.al.12}. 
Starting in $X(0)=0=Y(0)$, their $n$-th moment at time $t$ is given by
$$ E_0(X^n(t)) = E_0(Y^n(t)) =
1_{\mathbb N}(n/2) \frac{t^{n/2}}{(n/2)!}\prod_{k=1}^{n/2}\left(2^{2k-1}-1\right),\quad t\geq 0,
\quad n\in\mathbb N. $$
\end{ex}
\begin{proof}
Suppose that the truncation function $\chi$ is continuous and anti-symmetric. Define $q$ as in (\ref{e:ex1}). 
The function $q$ has an obvious entire extension. Define
\begin{eqnarray}
 b(x)&=&0,\nonumber\\
 c(x)&=& 1_{\{x=0\}},\nonumber\\
 F(x,\cdot)&:=&1_{\{x\neq0\}}\frac{\delta_x+\delta_{-x}}{2x^2}.\label{e:a}
\end{eqnarray}
Then $(b,c,F)$ is the corresponding triplet in the sense of Remark \ref{Bemerkung: triplet}. For any $n\in\mathbb N,k\in\mathbb R_+$ we also define
\begin{equation}\label{e:b}
q_{k,n}(x,u):=\begin{cases} q(x,u) & \text{ if }\vert x\vert\geq k2^{-n}, \\ 4^n\frac{\cos(uk2^{-n})-1}{k^2} & \text{otherwise.} \end{cases}
\end{equation}
Then $q_{k,n}$ is a continuous symbol and the associated linear operator is given by
$$\A _{k,n}f(x) = \begin{cases} \frac{f(2x)-2f(x)+f(0)}{2x^2} & \text{ if }\vert x\vert\geq k2^{-n}, \\ 4^n\frac{f(x+k2^{-n})-2f(x)+f(x-k2^{-n})}{2k^2} & \text{otherwise.} \end{cases}$$
The symbol $q_{k,n}$ satisfies the requirements of Theorem \ref{Satz: Existenz unter Stetigkeit} whence
there is a solution $X_{k,n}$ to the martingale problem  $(\A _{k,n},\delta_0)$ in the sense of \cite[Section 4.3]{ethier.kurtz.86}. 
Since 
\begin{eqnarray*}
X_{k,n}(t)&=&
X_{k,n}(0)+\int_0^tb(X_{k,n}(s))ds+X_{k,n}^c(t)\\
&&{}+(x-\chi(x))*\mu^{X_{k,n}}(t)+\chi(x)*(\mu^{X_{k,n}}-\nu^{X_{k,n}})(t)\\
&=&x*\mu^{X_{k,n}}(t)
=\sum_{s\leq t}\Delta X_{k,n}(s)
\end{eqnarray*}
and since
$\Delta X_{k,n}(t)\in\{\pm X_{k,n}(t-),\pm k2^{-n}\}$
by (\ref{e:a}, \ref{e:b}),
we conclude  
that $X_{k,n}$ takes values only in the closed set $M_k:=k\{\pm2^{z}:z\in\mathbb Z\}\cup\{0\}$. Moreover, for any real-valued Schwartz function $f$ we have uniform convergence
$\A _{k,n}f\rightarrow \A  f$ for $n\to\infty$,
where $\A$ denotes the operator associated with $q$.
Using the proof of Lemma \ref{l:zweites Moment} and Chebyshev's inequality,
we have that the sequence $(X_n)$ satisfies inequality (5.2) in \cite[Remark 4.5.2]{ethier.kurtz.86}.
The set of Schwartz functions is an algebra which separates points.
By \cite[Lemma 4.5.1, Remark 4.5.2]{ethier.kurtz.86} there is a subsequence of $(X_{k,n})_{n\in\mathbb N}$ which converges weakly to a solution $X_k$ of the martingale problem $(\A,\delta_0)$. Note that  $X_k$ takes values only in $M_k$. 

Both $X_1,X_{\sqrt 2}$ are solutions to the martingale problem related to $\A $ and initial law $\delta_0$, where $X_1$ takes values only in $M_1$ and $X_{\sqrt{2}}$ only in $M_{\sqrt{2}}$. 
Their extended generator in the sense of \cite[Definition 2.3]{cuchiero.al.12} 
is defined for all polynomials. Polynomials are mapped to polynomials of at most the same degree, which means that $X_1,X_{\sqrt 2}$ are polynomial processes by \cite[Theorem 2.10]{cuchiero.al.12}. \cite[Theorem 2.7]{cuchiero.al.12} and its proof 
yields the moments.
In particular, $E(X_1(t)^2)=E(X_{\sqrt{2}}(t)^2)>0$ for any  $t>0$, which means that 
$X_1(t)$, $X_{\sqrt{2}}(t)$ are not concentrated in zero.
Since $M_1\cap M_{\sqrt{2}}=\{0\}$, this implies that $X_1,X_{\sqrt{2}}$
cannot have the same law.
\end{proof}

\begin{figure}
\includegraphics[scale=0.8]{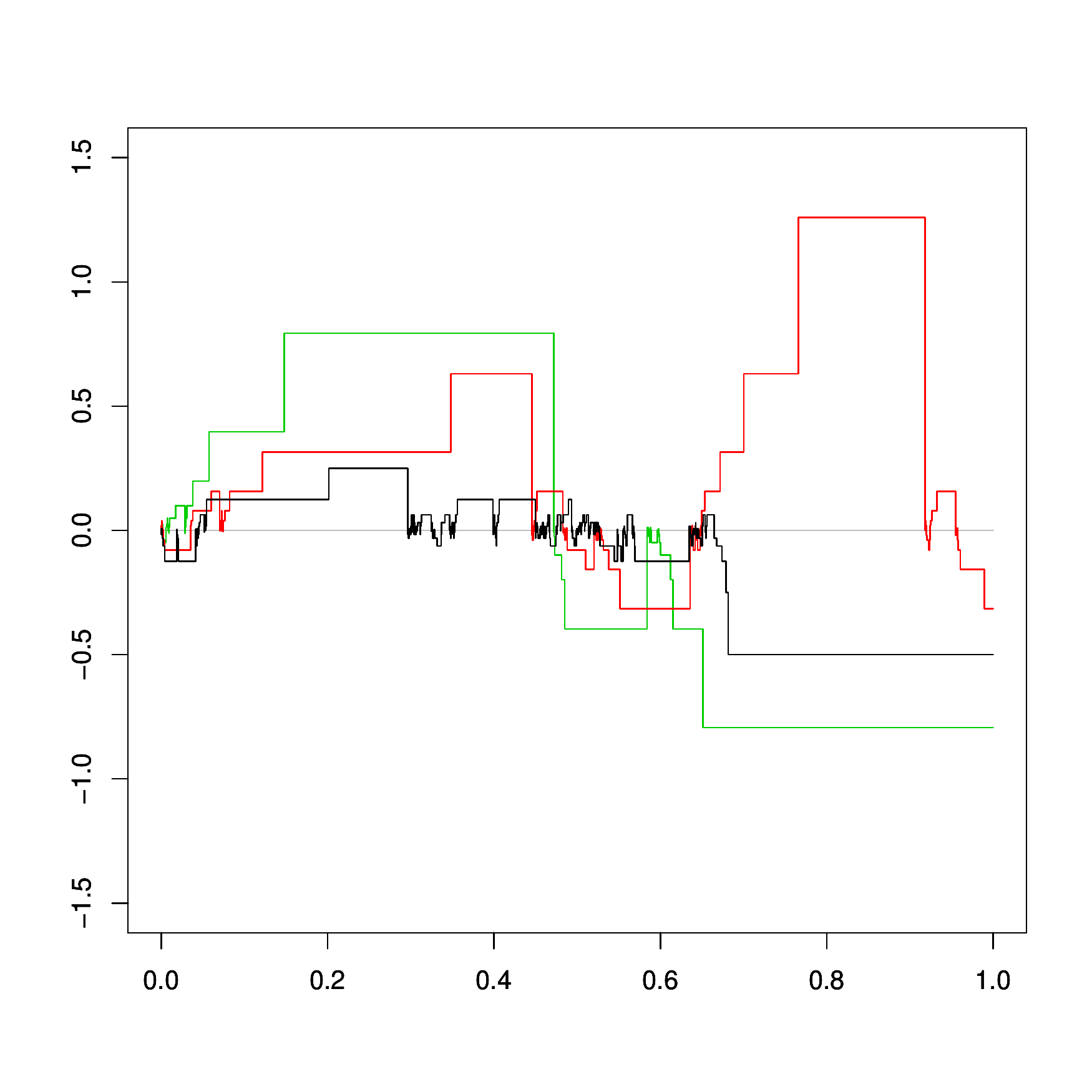}
\caption{The picture shows simulated paths from three different Markov processes with symbol 
(\ref{e:ex1}). All of them are based on the approximate generator $\A _{k,n}$ in the proof of Example \ref{Beispiel: analytisches Symbol ohne Eindeutigkeit} with $n=10$ and $k=1$ (black line), $k=\sqrt[3]{2}$ (red line) resp.\ $k=\sqrt[3]{4}$ (green line).}
\end{figure}

We now turn to a related example with analytic symbol where uniqueness fails.
It corresponds to an increasing process.
It is once more polynomial in the sense of 
\cite[Definition 2.3]{cuchiero.al.12}. 
However, since its state space equals $\mathbb R_+$, it is not perfectly in line with the setup of this paper.
\begin{ex}\label{Beispiel 2:analytisches symbol ohne Eindeutigkeit}
 Let 
 $$q:\mathbb R_+\times\mathbb R\rightarrow\mathbb C,\quad (x,u)\mapsto \begin{cases}\frac{e^{iux}-1}{x}&\text{for }x\neq0,\\iu&\text{otherwise.}\end{cases}$$ This function clearly has an entire extension to $\mathbb C\times\mathbb C$. 
 There are solutions $X,Y$ to the $q$-martingale problem which start in $0$ and are singular in the same sense as in the previous example.
 More generally, there are strong Markov processes $X$ and $Y$ with values in $\mathbb R_+$ and symbol $q$, which do not have the same law.
 Once more
$X$, $Y$ are polynomial processes. 
Starting in $0$, their $n$-th moment  at time $t$ is given by
 $$ E_0(X^n(t)) = E_0(Y^n(t)) =\frac{t^n}{n!}\prod_{k=1}^n\left(2^k-1\right),\quad t\geq0,n\in\mathbb N.$$
Exponent $q(x,\cdot)$ has L\'evy-Khintchine triplet $(\chi(x)/x,0,\delta_x/x)$ for $x>0$ and $(1,0,0)$ for $x=0$.
 Finally, observe that the continuous continuation given by
 $$ \tilde q(x,u) := 1_{\{x\geq 0\}}q(x,u)+1_{\{x<0\}}iu,\quad x,u\in\mathbb R$$
 on state space $\mathbb R$ yields the same process.
\end{ex}
\begin{figure}
\includegraphics[scale=0.8]{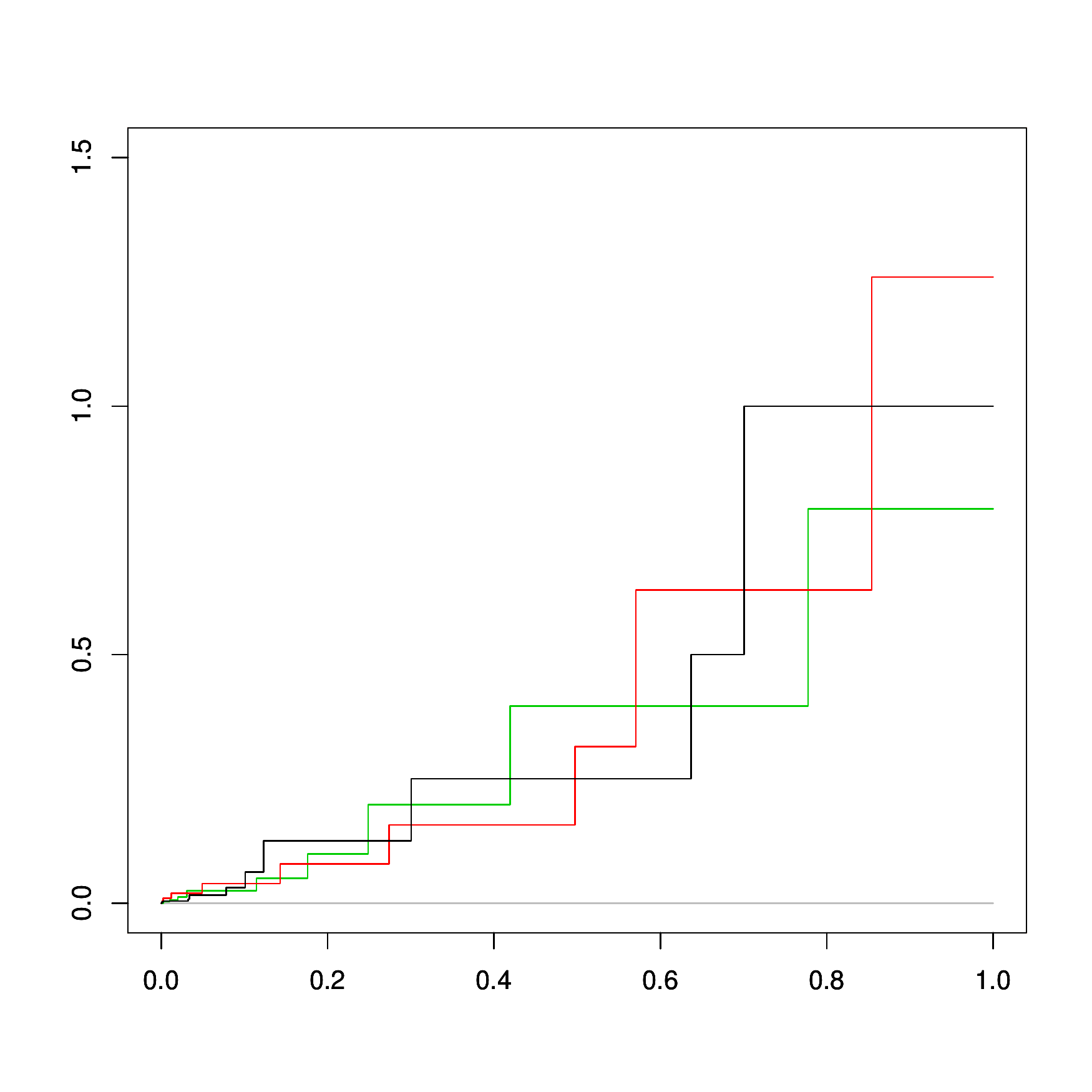}

The picture shows simulated paths from three different Markov processes. Each of them uses the approximate generator $\A _{k,n}$ appearing in the proof of Example \ref{Beispiel 2:analytisches symbol ohne Eindeutigkeit} with $k=1$ (black line), $k=\sqrt[3]{2}$ (red line) resp.\ $k=\sqrt[3]{4}$ (green line) and $n=10$ for all of them.
\end{figure}
\begin{proof}
 For any $n\in\mathbb N,k\in\mathbb R_+$ and $x,u\in\mathbb R$ we define
$$q_{k,n}(x,u):=\begin{cases} q(x,u) & \text{ for }x\in[ k2^{-n},k2^{n}], \\ 
q(k2^{-n},u) & \text{ for }x< k2^{-n}, \\
q(k2^{n},u) & \text{ for }x>k2^{n}. 
\end{cases}$$
Then $q_{k,n}$ is a continuous symbol with corresponding triplet
$(\chi(h(x))/h(x),0,\delta_{h(x)}/h(x))$
for $h(x):=(x\vee k2^{-n})\wedge k2^n$.
The associated linear operator is given by
$$\A _{k,n}f(x) = 
 \frac{f(x+h(x)) - f(x)}{h(x)} ,\quad x\in\mathbb R.$$
By Theorem \ref{Satz: Existenz unter Stetigkeit} there is a solution $X_{k,n}$ for the martingale problem related to $\A _{k,n}$ and initial law $\delta_0$. 
As in the previous example, we conclude that
$X_{k,n}(t)=\sum_{s\leq t}\Delta X_{k,n}(s)$ and that
$X_{k,n}$ takes values only in the closed set 
$M_k:=\{\ell k2^{n}:\ell\in\mathbb N\}\cup \{k2^{z}:z\in\mathbb Z\}\cup\{0\}$. 
Also as in Example \ref{Beispiel: analytisches Symbol ohne Eindeutigkeit},
for any real-valued Schwartz function $f$ on $\mathbb R$ we have uniform convergence
$\A _{k,n}f\rightarrow \A f$, where
$$\A(f):=
\begin{cases} \frac{f(2x) - f(x)}{x} & \text{ for }x>0, \\ 
f'(x) & \text{ for }x<0.
\end{cases}$$
Since
$$E\bigg(\sup_{t\in[0,T]}|X_{k,n}(t)|\bigg)
=E(X_{k,n}(T))\leq T,$$
Markov's inequality yields inequality (5.2) in \cite[Remark 4.5.2]{ethier.kurtz.86}.
Hence \cite[Lemma 4.5.1, Remark 4.5.2]{ethier.kurtz.86}
implies that there is a subsequence of $(X_{k,n})_{n\in\mathbb N}$ which converges weakly to a solution $X_k$ of the martingale problem $(\A,\delta_0)$. Obviously, $X_k$ takes values only in 
$M_k:=\{k2^{z}:z\in\mathbb Z\}\cup\{0\}$. 
Thus $X_1,X_{\sqrt 2}$ are solutions to the martingale problem related to $\A$ and initial law $\delta_0$, where $X_1$ takes values only in $M_1$ and $X_{\sqrt{2}}$ takes values only in $M_{\sqrt{2}}$. Since they are non-constant they cannot have the same distribution.

Observe that $\A $ satisfies the assumptions of \cite[Theorem 4.5.11(b)]{ethier.kurtz.86} because the real-valued Schwartz functions are an algebra that separates points and the second requirement is implied by \cite[Remark 4.5.12]{ethier.kurtz.86}. By \cite[Theorem 4.5.19]{ethier.kurtz.86}, $\A $ has different extensions $\A _1$, $\A _2$ which satisfy \cite[Theorem 4.5.19(a)]{ethier.kurtz.86} and correspond to generators of strong Markov processes by \cite[Theorem 4.5.19(d)]{ethier.kurtz.86}, cf.\ the proof of Theorem \ref{Satz: Existenz unter Stetigkeit} for detailed arguments. 

The statements on preservation of polynomials and the moments of $X$ follow as in the previous example.
\end{proof}

\section{The symbol and the uniqueness problem}\label{s:exuni}
The obvious question to ask is what conditions are needed to ensure uniqueness of the $q$-martingale problem for a given symbol $q$. For continuous processes the situation is well understood, The fact that SDE's have unique solutions under Lipschitz conditions directly yields uniqueness for $C^2$-symbols without jump part. 
\begin{prop}[no jumps]\label{Satz: SDE Bedingung}
 Let $q\in C^2(\mathbb R^d\times\mathbb R^d,\mathbb C)$ be a symbol which has the representation $q(x,u)=iub(x)-\frac{1}{2}uc(x)u$ for functions $b:\mathbb R^d\rightarrow\mathbb R^d,c:\mathbb R^d\rightarrow\mathbb R^{d\times d}$. Then uniqueness holds for the $q$-martingale problem.
\end{prop}
\begin{proof}
Observe that $c$ takes actually values in the set of positive semidefinite matrices. Let $\sigma(x)$ be the positive square root of $c(x)$ for all $x\in\mathbb R^d$. Then \cite[Theorem V.12.12]{rogers.williams.00} implies that $\sigma$ is Lipschitz.  \cite[Theorem V.12.1 and Section V.22]{rogers.williams.00} applied to $(b,\sigma)$ yields the claim.
\end{proof}

As Example \ref{Beispiel: analytisches Symbol ohne Eindeutigkeit} indicates, the previous theorem does not hold for processes with jumps. One could express $q$-martingale problems for general symbols $q$ in terms of an SDE, but it is not clear what conditions on $q$ warrant Lipschitz continuity of the corresponding coefficients. Stroock-Varadhan type results, however, require the diffusion part of the symbol to be non-singular. A systematic study of existence and uniqueness property has been undertaken by Hoh in a number of papers, see e.g.\ \cite[Theorem 5.7]{hoh.98}, \cite[Theorem 1.1]{hoh.00} and the related article \cite[Theorem 9.4]{jacob.93}.

One of the main contributions of the present paper is the following uniqueness result.
\begin{thm}[Uniqueness]\label{Satz: regulaeres Symbol}
 Let $q\in C(\mathbb R^d\times\mathbb R^d,\mathbb C)$ be a Hölder-continuous symbol 
satisfying the requirements of Theorem \ref{Satz: Existenz unter Stetigkeit}. Moreover, suppose that
there are $K\in\mathbb R_+$ and complex measures $(P_{t,u})_{t\in[0,1],u\in\mathbb R^d}$ such that 
\begin{enumerate}
 \item $\widehat{P_{t,u}}(x):=\int_{}e^{ivx}P_{t,u}(dv)=e^{tq(x,u)}$ for all $x\in\mathbb R^d$,
  \item $\int_{}\frac{1+\vert u+v\vert^2}{1+\vert u\vert^2}\vert P_{t,u}\vert(dv)\leq 1+Kt$.
\end{enumerate} 
 Then existence and uniqueness holds for the $q$-martingale problem. In particular, there is a unique Feller process starting in any given distribution and having symbol $q$.
\end{thm}
\begin{proof}
The proof is to be found in Section \ref{proofs2}.
\end{proof}

\begin{ex}
 Let $q(x,u) := (1-\cos(x))\psi(u)$, $x,u\in\mathbb R$ where $\psi$ is a real-valued L\'evy exponent on $\mathbb R$ which satisfies the requirements of Theorem \ref{Satz: Existenz unter Stetigkeit}, e.g.\ $\psi(u) = \frac{-u^2}{2}$. In particular, $q(0,u) = 0$ for any $u\in\mathbb R$. Then existence and uniqueness for the $q$-martingale problem hold. 
\end{ex}
\begin{proof}
Instead of verifying the conditions in the previous theorem directly, 
we refer to Lemma \ref{Satz: Eindeutigkeit unter Fourier Bedingung} below, which follows from Theorem \ref{Satz: regulaeres Symbol}.
\end{proof}

If the measure $P_{t,u}$ in Theorem \ref{Satz: regulaeres Symbol} happens to be nonnegative as in Example \ref{Beispiel: analytisches Symbol ohne Eindeutigkeit},  
we have 
$$\int\frac{1+\vert u+v\vert^2}{1+\vert u\vert^2} P_{t,u}(dv)\leq
1+{2|u|\over 1+|u|^2}\left|\int vP_{t,u}(dv)\right|
+{1\over 1+|u|^2}\int|v|^2 P_{t,u}(dv).$$
In this case condition (2) in Theorem \ref{Satz: regulaeres Symbol} 
can be interpreted as first and second moment condition on $P_{t,u}$,
which
can be vaguely viewed as a
``smoothness'' condition on $q$. However, in particular for complex $P_{t,u}$ it is less obvious how restrictive
condition (2) is and how one can verify it. We therefore provide a second uniqueness result which follows from
Theorem \ref{Satz: regulaeres Symbol}, but which is stated directly in terms of $q$.

\begin{thm}\label{Satz: Eindeutigkeit unter Elliptizitaet neu}
Let $q$ be a continuous symbol with $q(\cdot,u)\in C^{[d/2]+3}(\mathbb R^d,\mathbb C)$ for all $u\in\mathbb R^d$ and such that
$q$ satisfies the requirements of Theorem \ref{Satz: Existenz unter Stetigkeit},
 cf.\ also Corollary \ref{co:exuntermomenten}.
 Let  $\phi:\mathbb R^d\rightarrow\mathbb C$ be
a characteristic exponent satisfying the following conditions
\begin{align}
 \label{e:elliptizitaet} \vert \Re(q(x,u)) \vert \geq&\ g_1(x)\vert\phi(u)\vert, \\
 \label{e:wachstum} \vert \partial_x^\alpha q(x,u)\vert \leq&\ g_2(x)\vert\phi(u)\vert
\end{align}
for some bounded functions $g_1,g_2:\mathbb R^d\rightarrow(0,\infty)$ and any $\alpha\in\mathbb N^d$ with $\vert\alpha\vert\leq [d/2]+3$. Then existence and uniqueness holds for the $q$-martingale problem.
\end{thm}
\begin{proof}
The proof is to be found in Section \ref{proofs2}.
\end{proof}

Condition \eqref{e:wachstum} is a uniform smoothness requirement.
Condition \eqref{e:elliptizitaet}, however, means that the symbol is bounded from below in an appropriate sense. 
Such an ellipticity condition occurs in the Stroock-Varadhan existence and uniqueness result, cf.\ \cite[Theorem 4.3]{stroock.75}. 
The advantage of the result in \cite{stroock.75} is that continuity suffices and no extra smoothness is needed. Moreover, the drift only needs to be measurable. However, Stroock requires an ellipticity condition with respect to the explicit symbol $\phi(u)=-\frac{1}{2}u^2$, which means that a continuous diffusion part is present everywhere. His proof also uses some extra regularity for the jump measure which, however, could be relaxed.

Since Stroock and Varadhan have published their result, some variants of the Stroock-Varadhan theorem with a more general ellipticity condition have been established by several authors, i.e.\ with a more general function $\phi$ than for the original result, cf.\ \cite{jacob.schilling.01} for an overview. A recent result is due to B\"ottcher, cf.\ \cite[Theorem 2.8]{boettcher.05}, who requires equations (\ref{e:elliptizitaet},\ref{e:wachstum}) for arbitrary $\alpha$ and, moreover, a certain boundedness for the derivatives with respect to $u$. Theorem \ref{Satz: Eindeutigkeit unter Elliptizitaet neu} may be easier to apply in practice because it involves only finitely many derivatives and no smoothness in $u$.

\section{Proofs}\label{proofs}
\subsection{Proof of the existence theorem}\label{Abschnitt: Bew. Existenzsatz}
We start with the semimartingale property of solutions, cf.\ \cite[Corollary 2.6]{veerman.11} for a more general setup.
\begin{prop}[Semimartingale characteristics]\label{Satz: Semimartingaleigenschaft}
Let $q$ be a symbol and $X$  a solution to the $q$-martingale problem. Then $X$ is a semimartingale which allows for local or differential characteristics in the sense of \cite[Definition 2.2]{kallsen.04}. 
Moreover, $X$ is quasi-left continuous, i.e.\ $\Delta X_T=0$ almost surely for any finite predictable stopping time $T$.
If $q$ is represented by triplet  $(b,c,F)$ as in Definition \ref{Definition: Symbol},
\begin{equation}\label{e:tripel}
 (\omega,t)\mapsto\left(b(X(t-)),c(X(t-)),F(X(t-),\cdot)\right)(\omega)
\end{equation}
is a version of the local characteristics of $X$ relative to truncation function $\chi$.
\end{prop}
\begin{proof}
 Except for quasi-left continuity this follows from the definition and \cite[Theorem II.2.42]{js.87}. Quasi-left continuity is obtained from \cite[Proposition II.2.9(i)]{js.87}.
\end{proof}

We continue with a  
lemma  which yields a sufficient condition for the existence of second moments.
\begin{lem}\label{l:zweites Moment}
 Let $q$ be a symbol which satisfies the requirements of Theorem \ref{Satz: Existenz unter Stetigkeit}. Moreover, let $X$ be a solution to the $q$-martingale problem. Then 
 $E(\sup_{s\in[0,t]}|X(s)-X(0)|^2)<\infty$ for any $t\in\mathbb R_+$.
\end{lem}
\begin{proof}
Proposition \ref{Satz: Semimartingaleigenschaft} implies that $X$ is a semimartingale with local characteristics 
of the form (\ref{e:tripel}).
Boundedness of $g$ implies that it is a special semimartingale, cf.\ e.g.\ \cite[Proposition II.2.29(a)]{js.87}.
The finite variation part $A$ in its canonical decomposition $X=X(0)+M+A$
is of the form 
$$A(t)=\int_0^t\bigg( b(X(s-))+\int (x-\chi(x))F(X(s-),dx)\bigg)ds$$ 
and hence bounded on any compact interval.
For the local martingale part $M$ we have
$$\langle M_i,M_i\rangle(t)=\int_0^t \bigg(c_{ii}(X(s-))+\int x_i^2F(X(s-),dx)\bigg)ds,\quad i=1,\dots,d,$$
which is bounded on any compact interval as well.
Doob's inequality yields that
$$E\bigg(\sup_{s\in[0,t]}|M_i(s)|^2\bigg)\leq 4 E([M_i,M_i](t))= 4E( \langle M_i,M_i\rangle(t))<\infty,\quad i=1,\dots,d$$
for any $t\in\mathbb R_+$, which yields the claim.
\end{proof}

\begin{lem}\label{l:paulneu}
  Assume that the requirements of \cite[Theorem 4.5.19]{ethier.kurtz.86} are met where $\mathcal A_0$ denotes the corresponding operator on the state space $\mathbb R^d$ and $\mathcal T$ is the set of all bounded stopping times relative to the natural filtration on the Skorokhod space. 
 Then it is possible to choose measures $(P_x)_{x\in\mathbb R^d}$ on the Skorokhod space such that the canonical process $X$ is a solution to the $\mathcal A_0$-martingale problem with $X(0) = x$ a.s.\ under $P_x$, $x\in\mathbb R^d$ and such that $x\mapsto P_x(X(t)\in A)$ is measurable for any $t\geq 0$ and any Borel set $A\subset\mathbb R^d$. Moreover, $(X,(P_x)_{x\in\mathbb R^d})$ is strong Markov. 
Finally, if the $q$-martingale problem has several solutions for some initial law $P^{X(0)}=\nu$, then there are several such families of measures $(P_x)_{x\in\mathbb R^d}$.
\end{lem}
\begin{proof}
 {\em Step 1:}
 Let $\mathcal A$ be an extension of $\mathcal A_0$ as in \cite[Theorem 4.5.19]{ethier.kurtz.86}. Moreover, let $P_x \in \Gamma_{\delta_x}$ for $x\in \mathbb R^d$
 such that $P_x$ solves the $\A$-martingale problem, which exists by \cite[Theorem 4.5.19(c)]{ethier.kurtz.86}.
  We define the law $P(x,t,dy) := P_x(X(t) \in dy)$ for $x\in \mathbb R^d$, $t\geq 0$. Then $P(x,0,\cdot) = \delta_x$. 
Let $f:\mathbb R^d\to\mathbb R$ be continuous and bounded, which implies that 
$t\mapsto f(X(t))$ is right-continuous and bounded. 
By dominated convergence we have that $t\mapsto E_xf(X(t))$ is right-continuous as well. 
Since $x\mapsto E_x(f(X(t)))$ is measurable, $(x,t)\mapsto E_x(f(X(t))$ is measurable
by \cite[I.1.21 and I.1.26]{js.87}, applied to the right-continuous process $Y(x,t):=E_xf(X(t))$.
Since this holds for any continuous bounded $f$, we conclude that
$(x,t)\mapsto P(x,t,A)=P_x(X(t)\in A)$
is measurable for any $A\in\mathcal B(\mathbb R^d)$.

Let $s,t\geq 0$. \cite[Theorem 4.5.19(d)]{ethier.kurtz.86} yields
\begin{align*}
P(t+s,x,A)
&=  P_x( X(s+t)\in A)\\ 
&= E_x(E_x(1_A(X(s+t))|\mathcal F_t))\\
&= E_x(P_{X(t)}(X(s)\in A))\\
                                      &= \int P_y(X(s)\in A) P_x^{X(t)}(dy) \\
                                      &= \int P(y,s,A) P(x,t,dy)
\end{align*}
and hence $P$ is a transition function in the sense of \cite[Page 156]{ethier.kurtz.86}. 
  
{\em Step 2:}
From \cite[Theorem 4.5.19(d)]{ethier.kurtz.86} we also get that $X$ is a strong Markov process in the sense of \cite[Page 158]{ethier.kurtz.86} with transition function $P$.
 Indeed, let $P$ be a solution to the $\A$-martingale problem, $\tau$ a finite stopping time and $C\in\mathcal B(D_{\mathbb R^d}[0,\infty))$. 
 Then we have
$ E( 1_{\{X((\tau\wedge n)+\cdot)\in C\}}|\mathcal F_{\tau\wedge n}) = P_{X({\tau\wedge n})}(C) $
for any $n\in\mathbb N$
by \cite[Theorem 4.5.19(d)]{ethier.kurtz.86} because $\tau\wedge n$ is bounded.
For the strong Markov property we have to show that equality holds for $\tau$ instead of $\tau\wedge n$.
Clearly, $P_{X(\tau\wedge n)}(C)\to P_{X(\tau)}(C)$ pointwise for $n\rightarrow\infty$.
Moreover, $Y_n:=1_{\{X((\tau\wedge n)+\cdot)\in C\}} \rightarrow Y_\infty:=1_{\{X(\tau+\cdot)\in C\}}$ pointwise and hence in $L^2(P)$ for $n\rightarrow\infty$. We obtain
\begin{align*}
E( Y_n-Y_\infty |\mathcal F_{\tau\wedge n}) + E(Y_\infty|\mathcal F_{\tau\wedge n})
=E( Y_n|\mathcal F_{\tau\wedge n})
\to P_{X(\tau)}(C) 
\end{align*}
and $E( Y_n-Y_\infty |\mathcal F_{\tau\wedge n}) \rightarrow 0$ in $L^2(P)$ as $n\rightarrow \infty$ because
$$ E( |E( Y_n-Y_\infty |\mathcal F_{\tau\wedge n})|^2 ) \leq E(|Y_n-Y_\infty|^2) \rightarrow 0.$$
Moreover,
$E(Y_\infty|\mathcal F_{\tau\wedge n})= E(E(Y_\infty|\mathcal F_{\tau})|\mathcal F_{\tau\wedge n}) 
\rightarrow E(Y_\infty|\mathcal F_{\tau})$ in $L^2(P)$ when $n\rightarrow \infty$. 
Consequently,
 $$P_{X(\tau)}(C) = E(Y_\infty|\mathcal F_{\tau}).$$
 
{\em Step 3:} 
Now we turn to the last part of the claim. Here, we assume that $\Gamma_\nu$ contains more than one element for some Borel measure $\nu$ on $\mathbb R^d$. The proof of \cite[Lemma 4.5.19]{ethier.kurtz.86} actually constructs extensions $\mathcal A\not=\mathcal B$ given in \cite[Lemma 4.5.19(b)]{ethier.kurtz.86} which are maximal in the sense that any further extension of $\mathcal A$ or $\mathcal B$ does not meet \cite[Lemma 4.5.19(c)]{ethier.kurtz.86} any more.
  
Then, $\mathcal B \not\subseteq \mathcal A$ due to maximality. Thus there is $(f,g)\in\mathcal B$ with $(f,g)\notin\mathcal A$.
Assume by contradiction that $M(t) := f(X(t)) - \int_0^t g(X(s)) ds$ is a $P_x$-martingale for any $x\in \mathbb R^d$. Then $$\mathcal A_+ := \{(h+\lambda f,k+\lambda g):(h,k)\in\mathcal A,\lambda\in\mathbb R\}$$ is a strict extension of $\mathcal A$ which is a linear operator and such that the canonical process $X$ solves the $\mathcal A_+$-martingale problem under $P_x$ for any $x\in \mathbb R^d$. \cite[Proposition 4.3.5]{ethier.kurtz.86} yields that $\mathcal A_+$ is dissipative. This contradicts the maximality of $\mathcal A$.
We conclude that there is $x\in \mathbb R^d$ such that $M$ is not a $P_x$-martingale. 
However, 
$M$ is a $Q_x$-martingale if the canonical process $X$ is a solution 
to the $\mathcal B$-martingale problem under $Q_x$ for any $x\in \mathbb R^d$. Consequently, $P_x\neq Q_x$.
\end{proof}

\begin{proof}[Proof of Theorem \ref{Satz: Existenz unter Stetigkeit}]
{\em Step 1:}
Let $f$ be a real-valued Schwartz function and $\epsilon>0$.
Boundedness of $g$ implies
$$\sup_{x\in\mathbb R^d} F(x,\{y\in\mathbb R^d:\vert y\vert\geq a\}) < \epsilon$$
for some sufficiently large $a>0$
such that the support of $\chi$ is contained in $B(0,a)$.
Lemma \ref{l:Trippel Darstellung} yields
\begin{align*}
 \vert \A f(x)\vert &\leq \vert  \nabla f(x)\vert \left\vert b(x)+\int \left(y1_{B(0,a)}(y)-\chi(y)\right)F(x,dy)\right\vert \\
 &{}+ \frac{1}{2} \sup_{y\in B(x,a)}\vert Hf(y)\vert \left(\Tr(c(x))+\int_{B(0,a)} \vert y\vert^2 F(x,dy)\right) \\
  & + \int_{B(0,a)^c} \left\vert f(x+y)-f(x) \right\vert F(x,dy) \\
  &\leq \left(\vert \nabla f(x)\vert+\sup_{y\in B(x,a)}\vert Hf(y)\vert\right)g(x) + 2\sup_{y\in\mathbb R^d}\vert f(y)\vert \epsilon  \\
  &\underset{x\rightarrow\infty}{\rightarrow} 2\sup_{y\in\mathbb R^d}\vert f(y)\vert \epsilon
\end{align*}
for any $x\in\mathbb R^d$. Hence $\A f(x) \rightarrow 0$ for $x\rightarrow \infty$. 

Define $\tilde c(x):=c(x)+\int\chi(y)\chi(y)^\top F(x,dy)$. \cite[Theorem VII.2.9]{js.87} yields continuity of $b$, $\tilde c$ and $[\delta_{1,3}]$ in the sense of this theorem. 
Lemma \ref{l:Trippel Darstellung} and linearity of the trace imply
\begin{align*}
\A f(x) =&\ \nabla f(x)b(x) + \frac{1}{2}\mathrm{Tr}(Hf(x)\tilde c(x)) \\
  &+ \int_{} \left(f(x+y)-f(x)-\nabla f(x)\chi(y)- \frac{1}{2}\mathrm{Tr}(Hf(x)\chi(y)\chi(y)^T)\right)F(x,dy)
\end{align*}
for any $x\in\mathbb R^d$. Thus $x\mapsto \A f(x)$ is continuous. Together, this yields statement (3a).
Moreover, the indirect implication of statement (2) follows from Lemma  \ref{l:paulneu2}.

{\em Step 2:}
Let $f$ be any real-valued Schwartz function with maximum $x_0\in\mathbb R^d$, i.e.\ $\sup_{x\in\mathbb R^d}f(x)=f(x_0)$. Then $\nabla f(x_0)=0$ and $Hf(x_0)$ is negative semidefinite. Therefore
\begin{align*}
 \A f(x_0) =&\ \frac{1}{2}\mathrm{Tr}(Hf(x_0)c(x_0))+\int_{}\big(f(x_0+h)-f(x_0)\big)F(x_0,dh)\leq0,
\end{align*}
i.e., $\A$ satisfies the positive maximum principle in the sense of \cite[p.~165]{ethier.kurtz.86}, whence
statement (3b) holds.
Moreover, it is defined on a dense subset of $\hat C(\mathbb R^d)$
because its domain are the real-valued Schwartz functions.

{\em Step 3:}
Let $\phi:\mathbb R^d\rightarrow[0,1]$ be an infinitely differentiable function which is constant $1$ on the unit ball in $\mathbb R^d$ and whose support is contained in the centered ball with radius $2$. For any $n\in\mathbb N$ define the real-valued Schwartz function 
$$f_n:\mathbb R^d\rightarrow\mathbb R,\quad x\mapsto \phi(x/n).$$
Then $f_n\rightarrow 1$ pointwise and the second derivatives of $f_n$ are bounded by ${k}/{n^2}$, where $k$ is a common bound for the first two partial derivatives of $\phi$. Lemma \ref{l:Trippel Darstellung} yields
$$\A  f_n(x) = \int(f_n(x+y)-1)F(x,dy)$$
for $x\in\mathbb R^d$ and $n> \vert x\vert$. Due to a remainder estimate for the Taylor series we have
$$\vert f_n(x+y)-1\vert \leq \frac{k\vert y\vert^2}{2n^2}.$$
Thus the dominated convergence theorem yields
$$\vert \A  f_n(x)\vert\leq \int\frac{k\vert y\vert^2}{2n^2}F(x,dy)\rightarrow0,\quad n\rightarrow\infty,$$
whence $\A  f_n(x)\rightarrow 0$ pointwise. Similar arguments yield
$\vert \A  f_n(x)\vert\leq Kg(x)$
for any $x\in\mathbb R^d$, $n\in\mathbb N$ and some constant $K>0$ which does not depend on $x$ and $n$. Thus we have
$ f_n \rightarrow 1$ and $ \A  f_n \rightarrow 0$
for $n\rightarrow\infty$, where the convergence holds relative to the bp-topology, cf.\ \cite[p.~111]{ethier.kurtz.86}.
This implies statement (3c).
\cite[Theorem 4.5.4 and Remark 4.5.5]{ethier.kurtz.86} yield that for any probability measure $\mu$ on $\mathbb R^d$ there is
a solution to the martingale problem $(\A,\mu)$ in the sense of \cite[Section 4.3]{ethier.kurtz.86}.

{\em Step 4:}
In order to show that there are solutions to the $q$-martingale problem,
let $\mu$ be a probability measure and $X$ a solution to the martingale problem $(\A ,\mu)$
in the sense of \cite[Section 4.3]{ethier.kurtz.86}. 
Moreover, let $u\in\mathbb R^d$ and  $\phi:\mathbb R^d\rightarrow[0,1]$ as in Step 4.
Define 
$$ f_n(x) := e^{iux}\phi(x/n),\quad x\in\mathbb R^d,n\in\mathbb N.$$
Similarly as in Step 4 one shows that there is a bound $B<\infty$ such that $\vert f_n(x)\vert \leq 1$, $\vert \A f_n(x)\vert\leq B$ for any $x\in\mathbb R^d$, $n\in\mathbb N$, and
\begin{align*}
 f_n(x) &\rightarrow e^{iux},\\
 \A  f_n(x) &\rightarrow q(x,u)e^{iux}
\end{align*}
for any $x\in\mathbb R^d$. Thus 
 $$ M_{f_n}(t)\stackrel{n\to\infty}{\longrightarrow}M_u(t) := e^{iuX(t)} - \int_0^t e^{iuX(s)}q(X(s),u)ds$$
 a.s.\ for any $t\in\mathbb R_+$.
 By dominated convergence, $M_u$ is a martingale which shows that $X$ is a solution to the $q$-martingale problem.
Altogether, we obtain both the direct implication of statement (2) and statement (1).

 {\em Step 5:}
 The set of real-valued Schwartz functions is an algebra that separates points. Lemma \ref{l:zweites Moment} yields that that \cite[(4.5.33)]{ethier.kurtz.86} holds  for all choices of $K,\epsilon$ and $T$ where $K'$ can be chosen according to Chebyshev's inequality. Thus \cite[Theorem 4.5.11(b)]{ethier.kurtz.86} implies that the requirements of \cite[Theorem 4.5.19]{ethier.kurtz.86} are met for 
 $\Gamma:=\{P_\mu: P_\mu$ solves $(\A,\mu)$ for some probability measure $\mu$ on $\mathbb R^d\}$.
 Lemma \ref{l:paulneu} now yields statements (4) and (5).
 
 {\em Step 6:}
Let $x,u\in\mathbb R^d$ and $(P_x)_{x\in\mathbb R^d}$ measures on the canonical space such that the canonical process $X$ is a solution to the $q$-martingale problem with $X(0)=x$ $P_x$-a.s. for any $x\in\mathbb R^d$. Then
 $$ E_x(e^{iu X(t)}) =e^{iux} + E_x\left(\int_0^t e^{iu X(s)}q(X(s),u) ds\right),\quad t\geq 0.$$
 Hence right-continuity of $X$ in $0$ yields
 \begin{align*}
  \frac{E_x(e^{iu(X(t)-x)})-1}{t} &= \frac{1}{t}\int_0^t E_x(e^{iu(X(s)-x)}q(X(s),u) ds \\
                                 &\rightarrow E_x(e^{iu(X(0)-x)}q(X(0),u) \\
                                 &= q(x,u)
 \end{align*}
 for $t\downarrow0$, which is statement (6).
\end{proof}

\subsection{Proof of the uniqueness theorems}\label{proofs2}
The remainder of the paper is devoted to the proof of Theorems \ref{Satz: regulaeres Symbol} and  \ref{Satz: Eindeutigkeit unter Elliptizitaet neu}. 
 The idea  is as follows. We aim at proving uniqueness of univariate marginals in the Fourier domain, i.e.\ we show that for two solutions $X,Y$ the characteristic functions $\phi_{X(t)},\phi_{Y(t)}$ coincide. This will be done by a Gr\"onwall argument. We proceed in two steps. First we show that any solution can be approximated locally by a conditional L\'evy process, cf.\ Lemmas \ref{Lemma: Evolution der charakteristischen Funktion}, \ref{Lemma: Dynamik verschwindet}. Secondly we try to find bounds for the deviation rate of two piecewise L\'evy processes, which leads to Theorem \ref{Satz: regulaeres Symbol}. In order to derive Theorem \ref{Satz: Eindeutigkeit unter Elliptizitaet neu}, the conditions  are first verified for simple symbols. Moreover, the set of symbols meeting the requirements has a certain closedness property, cf.\ Lemma \ref{Proposition: Abgeschlossenheit}. Then we can deduce Lemma
\ref{Satz: Eindeutigkeit unter Fourier Bedingung} which states that uniqueness holds if the symbol can be locally approximated with a Fourier series satisfying some positivity condition. Finally, we construct such a Fourier series for elliptic symbols, cf.\ Lemma \ref{Satz: Eindeutigkeit unter Elliptizitaet}.

The localisation procedure for martingale problems reveals that uniqueness is a local property, cf.\ \cite[Section 4.6]{ethier.kurtz.86} and \cite[Theorem 3.28]{boettcher.al.13}. We restate a localisation theorem suitable for our applications with slightly different assumptions than the very related \cite[Theorem 3.28]{boettcher.al.13}; the proof however is basically the same.
\begin{prop}\label{S:Lokalisation}
 Let $q$ be a symbol such that existence holds for the $q$-martingale problem. Let $\U$ be an open covering for $\mathbb R^d$ and for all $U\in\U$ let $q_U$ be a symbol such that
\begin{enumerate}
 \item $q(x,u)=q_U(x,u)$ for any $x\in U,u\in\mathbb R^d$,
 \item existence and uniqueness holds for the $q_U$-martingale problem,
 \item $q(\cdot,u)$ is bounded for any $u\in\mathbb R^d$ and
 \item $q_U(\cdot,u)$ is bounded for any $u\in\mathbb R^d$, $U\in\mathcal U$.
\end{enumerate}
Then existence and uniqueness hold for the $q$-martingale problem.
\end{prop}
\begin{proof}
 W.l.o.g.\ we may assume that $\U$ is countable. Let $\mu$ be a probability measure on $\mathbb R^d$ and $U\in\U$. Define
   \begin{align*}
      \mathcal B := \{ (e^{iu\cdot},q(\cdot,u)e^{iu\cdot}): u\in\mathbb R^d \}, \\
      \mathcal B_U := \{ (e^{iu\cdot},q_U(\cdot,u)e^{iu\cdot}): u\in\mathbb R^d \}
   \end{align*}
  for any $U\in\mathcal U$. Observe that a stochastic c\`adl\`ag process is a solution to the $\mathcal B$-martingale problem if and only if it is a solution to the $q$-martingale problem. In particular, existence holds for the $\mathcal B$-martingale problem. Let $U\in\mathcal U$ and observe that a stochastic process is a solution to the stopped martingale problem $(\mathcal B,\mu,U)$ if and only if it is a solution to the stopped martingale problem $(\mathcal B_U,\mu,U)$
  in the sense of \cite[Page 216]{ethier.kurtz.86}. 
  
  Moreover, \cite[Theorem 4.6.1]{ethier.kurtz.86} yields that the stopped martingale problem $(\mathcal B_U,\mu,U)$ has unique solutions. Hence $(\mathcal B,\mu,U)$ has unique solutions for any $U\in\mathcal U$. \cite[Theorem 4.6.2]{ethier.kurtz.86} yields uniqueness for the martingale problem $(\mathcal B,\mu)$. Since $\mu$ was arbitrary, we have uniqueness for the $q$-martingale problem.
\end{proof}

The next result is a Gr\"onwall-type theorem with perturbation which will be useful later. 
\begin{lem}\label{Proposition: Groenwall}
Let $I=[0,T]$, $c\in(0,\infty)$, and $\beta:\mathbb R_+\rightarrow\mathbb R_+$ such that $\lim_{t\rightarrow0}{\beta(t)}/{t}=0$. Let $\phi:\mathbb R_+\rightarrow\mathbb R_+$ such that for all $s,t\in I$ with $s<t$ we have
$$\phi(t)\leq (1+(t-s)c)\phi(s)+\beta(t-s).$$
Then $\phi(t)\leq \phi(0)e^{ct}$ for all $t\in I$. In particular, $\phi=0$ if $\phi(0)=0$.
\end{lem}
\begin{proof}
 Let $t\in\mathbb R_+$ and $N\in\mathbb N$. The inequality above yields
$$\phi\left(t\frac{n+1}{N}\right)\leq (1+tc/N)\phi\left(t\frac{n}{N}\right)+\beta(t/N)$$
for $n=0,\dots,N-1$.
Hence 
$$\phi(t)\leq (1+tc/N)^N\phi(0)+\sum_{k=0}^{N-1}(1+tc/N)^k\beta(t/N).$$
Since $(1+tc/N)^N\leq\exp(tc)$, we have
$$\phi(t)\leq\exp(tc)\phi(0)+\sum_{k=0}^{N-1}(1+tc/N)^k\beta(t/N).$$
The geometric series sums up to
$$\sum_{k=0}^{N-1}(1+tc/N)^k=\frac{(1+tc/N)^{N}-1}{tc/N}\leq N\frac{e^{tc}-1}{tc}.$$
However, $N\beta(t/N)$ converges to $0$ for $N\rightarrow\infty$. Hence 
$$\sum_{k=0}^{N-1}(1+tc/N)^k\beta(t/N)\rightarrow0.$$
We conclude
$\phi(t)\leq\exp(tc)\phi(0)$
as desired.
\end{proof}

In the sequel we will work with the norm
$$\Vert\cdot\Vert:\hat C(\mathbb R^d,\mathbb C)\rightarrow\mathbb R_+,\quad \phi\mapsto\sup\left\{\frac{\vert\phi(u)\vert}{1+\vert u\vert^2}:u\in\mathbb R^d\right\}.$$
\begin{rem}
A sequence of characteristic functions which converges with respect to ${\Vert\cdot\Vert}$, converges uniformly on compact sets. L\'evy's continuity theorem \cite[Theorem 19.1]{jacod.protter.04} yields weak convergence of the corresponding sequence of random variables.
\end{rem}

As mentioned above, the proof of Theorem \ref{Satz: regulaeres Symbol} is based on local comparison to conditional L\'evy processes.
\begin{lem}[Comparison to conditional L\'evy process I]\label{Lemma: Evolution der charakteristischen Funktion}
 Let $q$ be a continuous symbol such that $q(\cdot,u)$ is bounded for all $u\in\mathbb R^d$. Moreover, let  $X$ be a solution to the $q$-martingale problem and for $s\geq 0$ let $Q_s$ be a regular version of the conditional law of $X$ given $\mathcal F_s$, i.e.\
  \begin{itemize}
   \item $Q_s:\Omega\times \mathcal B(D_{\mathbb R^d}[0,\infty))\rightarrow [0,1]$ is a transition kernel from $(\Omega,\mathcal F_s,P)$ to $\mathbb D:=D_{\mathbb R^d}[0,\infty)$ and
   \item \sloppy{$E( f(X)|\mathcal F_s) = \int f(\rho) Q_s(d\rho)$ for any bounded measurable function $f:\mathbb D\rightarrow\mathbb C$}.
  \end{itemize}
Define
$$\Gamma(t,s,u):=\int_s^t e^{(t-r)q(X(s),u)} \int_{\mathbb D} e^{iu\rho(r)} (q(\rho(r),u)-q(\rho(s),u))Q_s(d\rho) dr$$
for $s,t\in\mathbb R_+,u\in\mathbb R^d$. Then we have
$$\phi_X(t,u)=E\left(\exp((t-s)q(X(s),u)+iuX(s))\right)+E(\Gamma(t,s,u))$$
for any $t,s\in\mathbb R_+,u\in\mathbb R^d$ with $s<t$,
where $\phi_X(t,u):=E(e^{iuX(t)})$. Moreover,
$$\vert E\Gamma(t,s,u)\vert\leq \int_s^t E\vert q(X(r),u)-q(X(s),u)\vert dr$$
for any $x,u\in\mathbb R^d$, $0\leq s\leq t<\infty$.
\end{lem}
\begin{proof}
Let $s\in\mathbb R_+,u\in\mathbb R^d$. For all $t\in[s,\infty)$ we have
\begin{align*}
 \phi_s(t,u) &:= \int_{\mathbb D} e^{iu\rho(t)}Q_s(d\rho) \\
&= M_u(s)+\int_{\mathbb D} \int_0^tq(\rho(r),u)\exp(iu\rho(r))dr Q_s(d\rho)\\
&= M_u(s)+\int_0^sq(X(r),u)\exp(iuX(r))dr\\
&\ \ +\int_s^t \int_{\mathbb D} q(\rho(r),u)\exp(iu\rho(r))Q_s(d\rho) dr\\
&= e^{iuX(s)}+\int_s^t \int_{\mathbb D} q(\rho(r),u)\exp(iu \rho(r)) Q_s(d\rho) dr.
\end{align*}
We can see that $t\mapsto \phi_s(t,u)$ is $P$-a.s.\ continuous. Thus the canonical process  
on $(\mathbb D,\mathcal B(\mathbb D),Q_s(\omega,\cdot))$ is weakly continuous
for $P$-almost every $\omega\in\Omega$. Consequently,
 $$ r\mapsto \int_{\mathbb D} q(\rho(r),u)\exp(iu \rho(r)) Q_s(\omega,d\rho) $$
is continuous for $P$-almost every $\omega\in\Omega$. This shows that $t\mapsto \phi_s(t,u)$ is continuously differentiable for $P$-almost every $\omega\in\Omega$.

The fundamental theorem of calculus yields
\begin{align*}
\partial_t\phi_s(t,u) &= \int_{\mathbb D} q(\rho(t),u)\exp(iu\rho(t)) Q_s(d\rho)\\
 &= \phi_s(t,u)q(X(s),u) + g(s,t)
\end{align*}
 for all $t>s$ $P$-a.e.\ where 
   $$ g(s,t):= \int_{\mathbb D} \exp(iu\rho(t))(q(\rho(t),u)-q(\rho(s),u)) Q_s(d\rho).$$ 
 Moreover, we have $\phi_s(s,u)=e^{iuX(s)}$. The variation of constants formula \cite[Satz 98.5]{heuser.90a}
 implies
\begin{align*}
  \phi_s(t,u) &= e^{(t-s)q(X(s),u)+iuX(s)}+\int_s^t e^{(t-r)q(X(s),u)} g(s,r) dr \\
              &= e^{(t-s)q(X(s),u)+iuX(s)}+\Gamma(t,s,u).
\end{align*}
 Thus we obtain
 \begin{align*}
     E(e^{iuX(t)}) &= E(\phi_s(t,u)) \\
                    &= E\left( e^{(t-s)q(X(s),u)+iuX(s)} \right) + E(\Gamma(t,s,u)).
  \end{align*}
 Finally, we have
  \begin{align*}
    \vert E(\Gamma(t,s,u)) \vert &\leq E\left( \int_s^t \int_{\mathbb D} \vert (q(\rho(r),u)-q(\rho(s),u)) \vert Q_s(d\rho)dr \right) \\
       &= \int_s^t E\left( \vert (q(X(r),u)-q(X(s),u)) \vert \right) dr
  \end{align*}
  for any $t\geq s$.
\end{proof}

\begin{lem}[Comparison to conditional L\'evy process II]\label{Lemma: Dynamik verschwindet}
 Let $q$ be a Hölder continuous symbol. 
 Moreover, let $I\subset\mathbb R_+$ be a bounded interval and $X$ a solution to the $q$-martingale problem. Then there is a function $\beta:\mathbb R_+\rightarrow\mathbb R_+$ such that $\lim_{t\rightarrow 0}\beta(t)/t=0$ and 
$$\int_s^t\frac{E\vert q(X(r),u)-q(X(s),u)\vert}{1+\vert u\vert^2}dr\leq \beta(t-s)$$
for all $s,t\in I,u\in\mathbb R^d$ with $s<t$.
\end{lem}
\begin{proof}
 Let $f$ be a bounded and continuous function such that
$$\vert q(x,u)-q(y,u)\vert\leq f(x-y)(1+\vert u\vert^2),\quad x,y,u\in\mathbb R^d.$$
Then 
$$\frac{E\vert q(X(r),u)-q(X(s),u)\vert}{1+\vert u\vert^2}\leq E(f(X(r)-X(s))).$$
Proposition \ref{Satz: Semimartingaleigenschaft} states that $X$ is quasi-left continuous. Hence
$$H:\mathbb R_+\times\mathbb R_+\rightarrow\mathbb R_+,(r,s)\rightarrow E(f(X(r)-X(s)))$$
is continuous and $H(s,s)=0$ for all $s\in\mathbb R_+$. The mean value theorem theorem yields the claim
for $\beta(t):=t\sup\{H(r,s):r,s\in I$ and $|r-s|\leq t\}$.
\end{proof}

We can now show that the univariate marginals of solutions 
to the martingale problem are uniquely determined under certain conditions.
\begin{lem}\label{Proposition: 1d Eindeutigkeitssatz}
 Let $q\in C(\mathbb R^d\times\mathbb R^d,\mathbb C)$ be a continuous and Hölder-continuous symbol. Moreover, let $K\in\mathbb R_+$ and $I=[0,t_{0}]$ for some $t_0>0$. Assume that for any $t\in I,u\in\mathbb R^d$ there is a complex measure $P_{t,u}$ such that
\begin{enumerate}
 \item $\widehat{P_{t,u}}(x)=e^{tq(x,u)}$ for all $x\in\mathbb R^d$ and
 \item $\int_{}\frac{1+\vert u+v\vert^2}{1+\vert u\vert^2}\vert P_{t,u}\vert(dv)\leq 1+Kt$.
\end{enumerate}
If $X,Y$ are solutions to the $q$-martingale problem with the same initial distribution, $X(t)$ and $Y(t)$ have the same distribution for all $t\in I$.
\end{lem}
\begin{proof} 
Observe that condition (2) implies that the total variation measure $\vert P_{t,u}\vert$ is finite. Define $d(t):=\Vert \phi_X(t,\cdot)-\phi_Y(t,\cdot)\Vert$ for all $t\in\mathbb R_+$ where $\phi_X(t,\cdot)$ and $\phi_Y(t,\cdot)$ denote the characteristic functions of $X(t)$ resp.\ $Y(t)$. Let $g_{t,u}(x):=e^{tq(x,u)+iux}$. Lemmas \ref{Lemma: Evolution der charakteristischen Funktion}, \ref{Lemma: Dynamik verschwindet} yield
$$d(t)\leq \sup\left\{\frac{\vert E\left(g_{t-s,u}(X(s))\right)-E\left(g_{t-s,u}(Y(s))\right)\vert}{1+\vert u\vert^2}:u\in\mathbb R^d\right\}+\beta(t-s)$$
for all $s,t\in I$ with $s<t$, where $\beta$ is a function as in Lemma \ref{Lemma: Dynamik verschwindet} with $\lim_{t\rightarrow0}\beta(t)/t=0$. Moreover, condition (1) and Fubini's theorem imply
\begin{eqnarray*}
E\left(g_{t-s,u}(X(s))\right)
&=&\int\int e^{ivX(s)}P_{t-s,u}(dv)e^{iuX(s)}dP\\
&=&\int_{}\phi_X(s,u+v)P_{t-s,u}(dv)
\end{eqnarray*}
and likewise for $Y$.
We obtain
\begin{align*}
d(t)\leq&\ d(s)\sup\left\{\int_{}\frac{1+\vert u+v\vert^2}{1+\vert u\vert^2}\vert P_{t-s,u}\vert(dv):u\in\mathbb R^d\right\}+\beta(t-s)\\ 
  \leq&\ d(s)(1+K(t-s))+\beta(t-s).
\end{align*}
By Lemma \ref{Proposition: Groenwall}  we have $d(t)=0$ for all $t\in I$. Thus the characteristic functions of $X(t)$ and $Y(t)$ coincide, whence  they have the same law.
\end{proof}

\begin{cor}\label{Satz: Eindeutigkeitssatz}
 Assume that the requirements of Lemma \ref{Proposition: 1d Eindeutigkeitssatz} are fulfilled and that existence holds for the $q$-martingale problem. Then uniqueness holds for the $q$-martingale problem.
\end{cor}
\begin{proof}
Let $X,Y$ be solutions with the same initial law and let  $T\in[0,\infty]$ be maximal such that 
$X(t),Y(t)$ have the same law for all $t\in[0,T)$.
By Proposition \ref{Satz: Semimartingaleigenschaft} $X$ and $Y$ are quasi-left continuous, which implies that
$X(T)$ and $Y(T)$ have the same law.
Assume by contradiction that $T\neq \infty$. 
Then $\widetilde X(t) := X(T+t)$, $\widetilde Y(t) := Y(T+t)$ are solutions to the $q$-martingale problem with the same initial law. Lemma \ref{Proposition: 1d Eindeutigkeitssatz} yields that $\widetilde X$, $\widetilde Y$ have the same one-dimensional distribution up to $t_0$ and hence $X$, $Y$ have the same univariate marginals up to $T+t_0$. This contradicts the maximality of $T$. Hence \cite[Theorem 4.4.2]{ethier.kurtz.86} yields the claim.
\end{proof}

\begin{proof}[Proof of Theorem \ref{Satz: regulaeres Symbol}]
 Theorem \ref{Satz: Existenz unter Stetigkeit} implies existence and Lemma \ref{Satz: Eindeutigkeitssatz} yields uniqueness.
 The second statement follows from Proposition \ref{p:feller}.
\end{proof}

We now turn to the proof of Theorem \ref{Satz: Eindeutigkeit unter Elliptizitaet neu}. 
\begin{lem}\label{Proposition: Abgeschlossenheit}
Let $q$ be a symbol. Consider $q_n:\mathbb R^d\times\mathbb R^d\rightarrow\mathbb C$, 
complex measures $P_{t,u,n}$  on $\mathbb R^d$, and $K_n\geq 0$ such that
\begin{enumerate}
 \item $\widehat{P}_{t,u,n}(x) = \exp(tq_n(x,u))$,
 \item $|P_{t,u,n}|(\mathbb R^d) \leq 1$,
 \item $\int v |P_{t,u,n}|(dv) = 0$,
 \item $\int |v|^2 |P_{t,u,n}|(dv) \leq tK_n(1+|u|^2)$,
 \item $q(x,u) = \sum_{n\in\mathbb N}q_n(x,u)$, and
 \item $K:=\sum_{n=1}^\infty K_n <\infty$
\end{enumerate}
for any $n\in\mathbb N$, $x,u\in\mathbb R^d$, $t\in[0,1]$. Then there are complex measures $P_{t,u}$ on $\mathbb R^d$ such that
\begin{enumerate}
 \item $\widehat{P}_{t,u}(x) = \exp(tq(x,u)),$
 \item $\int \frac{1+|u+v|^2}{1+|u|^2} |P_{t,u}|(dv) \leq 1+Kt$
\end{enumerate}
 for any $x,u\in\mathbb R^d$, $t\in[0,1]$. 
\end{lem}
\begin{proof}
  Let $t\in[0,1]$, $u\in\mathbb R^d$. Observe that $(\otimes_{n=1}^l|P_{t,u,n}|((\mathbb R^d)^l))_{l\in\mathbb N}$ is a decreasing sequence and denote its limit by $c\in[0,1]$. Note that
   \begin{align*}
     c &= \lim_{l\rightarrow\infty}\prod_{n=1}^l|P_{t,u,n}|(\mathbb R^d) \\
       &\geq \bigg|\lim_{l\rightarrow\infty}\prod_{n=1}^lP_{t,u,n}(\mathbb R^d)\bigg| \\
       &= \bigg|\lim_{l\rightarrow\infty} \exp\Big(t\sum_{n=1}^lq_n(x,u)\Big)\bigg| \\
       &= \exp(t \Re(q(x,u))) \\
       &>0.
\end{align*}    
For $a_n:=P_{t,u,n}(\mathbb R^d) = \exp(tq_n(0,u))$, assumption (5) yields $\Pi_{n=1}^\infty a_n = \exp(tq(0,u))\in\mathbb C$.

By (4), (6) and Proposition \ref{p:infinite convolution} the infinite convolutions $P_{t,u}$ of $(P_{t,u,n})_{n\in\mathbb N}$ and
$Q_{t,u}$ of $(|P_{t,u,n}|)_{n\in\mathbb N}$ exist and we have 
$|P_{t,u}|\leq Q_{t,u}$
in the sense that the density is bounded by one. Moreover,
\begin{align*}
  \int e^{ivx} P_{t,u}(dv) &= \lim_{l\rightarrow\infty} \int e^{ivx} \left(P_{t,u,1}*\dots*P_{t,u,l}\right)(dv) \\
    &= \lim_{l\rightarrow\infty} \exp\left(t\sum_{n=1}^lq_n(x,u)\right) \\
    &= \exp(tq(x,u))
\end{align*}
for any $x\in\mathbb R^d$. Hence, $P_{t,u}$ satisfies (1).

By Proposition \ref{p:unend product compl} the infinite product measure $\overline P_{t,u}$ of $(P_{t,u,n})_{n\in\mathbb N}$
exists.
Let $\pi_n:(\mathbb R^d)^{\mathbb N}\rightarrow\mathbb R^d,(x_l)_{l\in\mathbb N}\mapsto x_n$. Then we have
 \begin{align*}
  \left|\int |\pi_n|^2 d\overline P_{t,u}\right| &\leq \int |\pi_n|^2 d|\overline P_{t,u}| \\
  &= \frac{c}{|P_{t,u,n}|(\mathbb R^d)}\int v^2 |P_{t,u,n}|(dv) \\
  &\leq tK_n(1+|u|^2) 
\end{align*}  
and
 \begin{align*}
\int \pi_n\pi_m d|\overline P_{t,u}| &= \frac{c}{|P_{t,u,n}|(\mathbb R^d)|P_{t,u,m}|(\mathbb R^d)} \int v |P_{t,u,n}|(dv) \int w |P_{t,u,m}|(dw) \\
    &= 0,
\end{align*}  
 where we used (3) and (4).
We conclude that
   \begin{align*}
     \int \bigg|\sum_{n=1}^l\pi_n\bigg|^2 d|\overline P_{t,u}| &= \int \sum_{n=1}^l|\pi_n|^2 d|\overline P_{t,u}| \\
      &\leq \sum_{n=1}^l tK_n(1+|u|^2) \\
      &\leq tK(1+|u|^2).
   \end{align*}
 This implies $$\int v^2 Q_{t,u}(dv) = \lim_{l\rightarrow\infty} \int \bigg(\sum_{n=1}^l\pi_n\bigg)^2 d|\overline P_{t,u}| \leq tK(1+|u|^2). $$
Similar arguments show that $Q_{t,u}(\mathbb R^d)\leq 1$ and $\int v Q_{t,u}(dv) = 0$.    
Thus we have
 \begin{align*}
   \int \frac{1+|u+v|^2}{1+|u|^2} |P_{t,u}|(dv) &\leq \int \frac{1+|u+v|^2}{1+|u|^2} Q_{t,u}(dv) \\
    &\leq 1 + tK
 \end{align*}
 as desired.
\end{proof}
 
 Recall from Theorem \ref{Satz: regulaeres Symbol} 
 that uniqueness holds for symbol $q$ in Lemma \ref{Proposition: Abgeschlossenheit}
 if it satisfies the requirements of Theorem \ref{Satz: Existenz unter Stetigkeit}.
 The functions $q_n$ will later be chosen from the following lemma.
\begin{lem}\label{Lemma: Existenz eines Masses}
 Let $n\in\mathbb R^d$, $a,b\in\mathbb C$, $k\in\mathbb R$ such that $\Re(a)\geq\vert b\vert$. Then there is a complex measure $P_t$ for any $t\in[0,1]$ such that
\begin{enumerate}
 \item $\widehat{P_{t}}(x) = \exp(t(b\cos(knx)-a))$,
 \item $\vert P_{t}\vert(\mathbb R^d) \leq 1$,
 \item $\int_{}v\vert P_{t}\vert(dv) =0$,
 \item $\int_{}\vert v\vert^2\vert P_{t}\vert(dv) \leq t\vert b\vert \vert kn\vert^2$.
\end{enumerate}
Moreover, there is a complex measure $Q_t$ such that
\begin{enumerate}
 \item $\widehat{Q_{t}}(x) = \exp(t(b\sin(knx)-a))$,
 \item $\vert Q_{t}\vert(\mathbb R^d) \leq 1$,
 \item $\int_{} v\vert Q_{t}\vert(dv) = 0$,
 \item $\int_{}\vert v\vert^2\vert Q_{t}\vert(dv) \leq t\vert b\vert \vert kn\vert^2.$
\end{enumerate}
\end{lem}
\begin{proof}
 Let $\mu$ be a complex measure on $\mathbb R^d$ with total variation less or equal $1$. 
 Moreover, let $P_t:=\exp(-ta)\exp(tb\mu)$, cf.\ Appendix \ref{s:appendix}. Then we have for all $x\in\mathbb R^d$
\begin{align}
\widehat P_t(x) =&\ \exp(-ta)\exp(tb\hat\mu(x)),\notag\\
\vert P_t\vert \leq&\ \exp(-t\Re(a))\exp(t\vert b\vert\vert\mu\vert),\label{second} \\
\vert P_{t}\vert(\mathbb R^d) \leq&\ \exp(t(\vert b\vert-\Re(a)))\leq 1,\notag\\
\int_{}\vert v\vert^2\vert P_{t}\vert(dv) \leq&\ t\vert b\vert \left(\int\vert v\vert^2\vert\mu\vert(dv)
+2\bigg|\int v|\mu|(dv) \bigg|^2\right)\notag
\end{align}
for any $x\in\mathbb R^d$, where \eqref{second} means that the measure on the left is absolutely continuous with density at most one relative to the measure on the right.  The last inequality follows from Remark \ref{b:Momente}. For the specific choice $\mu=\frac{1}{2}(\delta_{nk}+\delta_{-nk})$,  Lemma \ref{l:symmetrie der variation} yields
$$\int_{}v\vert P_{t}\vert(dv)=0,$$
which shows the first assertion. For the specific choice $\mu=\frac{1}{2i}(\delta_{nk}-\delta_{-nk})$, Lemma \ref{l:spezielle Symmetrie} implies
 $$\int_{}v\vert P_{t}\vert(dv) =0$$
 and hence  the second claim.
\end{proof}

\begin{lem}\label{Lemma: fHoelder stetig}
Let $q\in C^{(1,0)}(\mathbb R^d\times\mathbb R^d,\mathbb C)$ be a symbol such that $(x,u)\mapsto \frac{q(x,u)}{1+|u|^2}$ and $(x,u)\mapsto \frac{\nabla_1q(x,u)}{1+|u|^2}$ are bounded.
Then $q$ is ($f$-)H\"older continuous for
 $$f(x) := \min\left\{2\sup_{y\in\mathbb R^d}\frac{|q(y,u)|}{1+|u|^2},|x|\sup_{y\in\mathbb R^d}\frac{|\nabla_1q(y,u)|}{1+|u|^2}\right\},\quad x\in\mathbb R^d.$$
In particular, if $q$ satisfies the requirements of Theorem \ref{Satz: Eindeutigkeit unter Elliptizitaet neu} for some $\phi$, then $q$ is H\"older-continuous.
\end{lem}
\begin{proof}
  The first part follows from the mean value theorem.
  
  Now assume that $q$ satisfies the requirements of Theorem \ref{Satz: Eindeutigkeit unter Elliptizitaet neu}. Equation \eqref{e:wachstum} yields that $|q(x,u)| \leq g_2(x)\phi(u)$ for some continuous function $g_2:\mathbb  R^d\rightarrow \mathbb R_+$ which is bounded by some constant $C_1<\infty$. Since $|\phi(u)|\leq C_2(1+|u|^2)$ for some constant $C_2<\infty$, we get $|q(x,u)|\leq C_1C_2(1+|u|^2)$. Moreover, we have $|\partial_{x_j}q(x,u)| \leq g_2(x)\phi(u)$ and hence $(x,u)\mapsto \frac{\nabla_1q(x,u)}{1+|u|^2}$ is bounded by $C(1+|u|)^2$ where $C:=dC_1C_2$.
\end{proof}

\begin{lem}[Fourier conditions]\label{Satz: Eindeutigkeit unter Fourier Bedingung}
 Let $q\in C^{(1,0)}(\mathbb R^d\times\mathbb R^d,\mathbb C)$ be a symbol with the following properties.
 \begin{enumerate}
  \item  
$q$ satisfies the requirements of Theorem \ref{Satz: Existenz unter Stetigkeit}.
\item There is a constant $c>0$ such that $|q(x,u)| + \vert \nabla_1q(x,u)\vert\leq c(1+\vert u\vert^2)$ for all $x,u\in\mathbb R^d$.
\item It has Fourier series representation, i.e.\ there are $a_n(u),b_n(u)\in\mathbb C$ for all $n\in\mathbb Z^d$ and a constant $k>0$ such that
\begin{equation}\label{e:sincos}
 q(x,u) = \sum_{n\in\mathbb Z^d} \left(a_n(u)\cos(knx)+b_n(u)\sin(knx)\right)
\end{equation}
    and the family $(a_n,b_n)_{n\in\mathbb Z^d}$ satisfies:
 \begin{enumerate}
   \item the real part of $-a_0(u)$ dominates the absolute sum of the other coefficients, i.e.\ 
\begin{equation}\label{e:fc1}
  -\Re(a_0(u))\geq \sum_{n\in\mathbb Z^d\backslash\{0\}}\left(\vert a_n(u)\vert+\vert b_n(u)\vert\right)
\end{equation}
 for all $u\in\mathbb R^d$ and
 \item 
 \begin{equation}\label{e:fc2}
  K:=|k|^2\sup\left\{\sum_{n\in\mathbb Z^d}\vert n\vert^2\left(\frac{\vert a_n(u)\vert}{1+\vert u\vert^2}+\frac{\vert b_n(u)\vert}{1+\vert u\vert^2}\right):u\in\mathbb R^d\right\}<\infty.
 \end{equation}
\end{enumerate}
 \end{enumerate}
Then existence and uniqueness holds for the $q$-martingale problem.
\end{lem}
\begin{proof}
Lemma  \ref{Lemma: fHoelder stetig} states that $q$ is H\"older continuous.

Since the coefficient $b_0$ does not play any role in the representation of $q$ we may assume $b_0=0$. Let $u\in\mathbb R^d,t\in[0,1]$. Then the Fourier series can be rewritten as
$$q(x,u) = \tilde a_0(u) + \sum_{n\in\mathbb Z^d\backslash\{0\}}\Big(\left(a_n(u)\cos(knx)-\vert a_n(u)\vert)+(b_n(u)\sin(knx)-\vert b_n(u)\vert\right)\Big),$$
where $\tilde a_0(u):= a_0(u)+\sum_{n\in\mathbb Z^d\backslash\{0\}}(\vert a_n(u)\vert+\vert b_n(u)\vert)$ and $\Re(\tilde a_0(u))\leq 0$. By Lemma \ref{Lemma: Existenz eines Masses} there are complex measures $P_{t,u,n},Q_{t,u,n}$, such that
\begin{enumerate}
 \item $\widehat{P_{t,u,n}}(x) = \exp(t(a_n(u)\cos(knx)-\vert a_n(u)\vert))$,
 \item $\vert P_{t,u,n}\vert(\mathbb R^d) \leq 1$,
 \item $\int_{} v\vert P_{t,u,n}\vert(dv) = 0$,
 \item $\int_{}\vert v\vert^2\vert P_{t,u,n}\vert(dv) \leq t\vert a_n(u)\vert \vert kn\vert^2$
\end{enumerate}
and
\begin{enumerate}
 \item $\widehat{Q_{t,u,n}}(x) = \exp(t(b_n(u)\sin(knx)-\vert b_n(u)\vert))$,
 \item $\vert Q_{t,u,n}\vert(\mathbb R^d) \leq 1$,
 \item $\int_{} v\vert Q_{t,u,n}\vert(dv) = 0$,
 \item $\int_{}\vert v\vert^2\vert Q_{t,u,n}\vert(dv) \leq t\vert b_n(u)\vert \vert kn\vert^2$
\end{enumerate}
for all $n\in\mathbb Z^d\setminus\{0\}$. 
Moreover, 
the measure $P_{t,u,0}:=\exp(t\tilde a_0(u))\delta_0$ satisfies
\begin{enumerate}
 \item $\widehat{P_{t,u,0}}(x) = \exp(t\tilde a_0(u))$,
 \item $\vert P_{t,u,0}\vert(\mathbb R^d)= \exp(t\Re(\tilde a_0(u)))\leq 1$,
 \item $\int_{} v\vert P_{t,u,0}\vert(dv) = 0$,
 \item $\int_{}\vert v\vert^2\vert P_{t,u,0}\vert(dv) = 0.$
\end{enumerate}
Lemmas \ref{Proposition: Abgeschlossenheit} and \ref{Proposition: 1d Eindeutigkeitssatz} yield uniqueness of the solution to the 
$q$-martingale problem.
\end{proof}

The Fourier conditions in Lemma \ref{Satz: Eindeutigkeit unter Fourier Bedingung} might seem hard to verify. However, ellipticity and Fourier ellipticity are almost equivalent as can be seen from the proof of the next lemma.
\begin{lem}\label{Satz: Eindeutigkeit unter Elliptizitaet}
 Let $q$ be a continuous symbol such that $q(\cdot,u)\in C^{[d/2]+3}(\mathbb R^d,\mathbb C)$ for all $u\in\mathbb R^d$ and such that
\begin{enumerate}
 \item $q$ satisfies the requirements of the existence theorem \ref{Satz: Existenz unter Stetigkeit} and
 \item for every $x_0$ there is a neighbourhood $V$ of $x_0$ and $L<\infty$ such that $\psi:=q(x_0,\cdot)$ satisfies
\begin{eqnarray}
\vert \partial_x^\beta q(x,u)\vert &\leq& L\vert\Re(\psi(u))\vert,\label{e:20a}\\
\vert \partial_x^\alpha q(x,u)\vert &\leq& L (1+\vert u\vert^2)\label{e:20b}
\end{eqnarray}
for all $x\in V,u\in\mathbb R^d,\beta,\alpha\in\mathbb N^d$ with $\vert\beta\vert\leq [d/2]+1$ 
and $\vert\alpha\vert\leq [d/2]+3$.\end{enumerate}
Then existence and uniqueness hold for the $q$-martingale problem.
\end{lem}
\begin{proof}
{\em Step 1:}
Set $m:=[d/2]+1$ 
and let $u\in\mathbb R^d$. W.l.o.g.\ we may assume that $V$ is open, convex, and bounded.
The first inequality together with the mean value theorem implies
  $$\vert q(x,u)-q(x_0,u)\vert\leq L\vert x-x_0\vert\vert\Re(\psi(u))\vert$$ 
for all $x\in V$. Let $\phi\in C^\infty([0,1]^d,[0,1])$ such that it is constant $1$ on $[1/4,3/4]^d$ 
and compactly supported in $(0,1)^d$. Set 
$$C_\phi:=\sup\{\vert\partial^\alpha\phi(x)\vert:\alpha\in\mathbb N^d,\vert\alpha\vert\leq m+2,x\in\mathbb R^d\},$$ 
let $\ell\geq1 $ be large enough such that the cube centered at $x_0$ of radius $1/\ell$ is contained in $V$, and define
$$q_\ell:[0,1]^d\times\mathbb R^d\rightarrow\mathbb C,(y,u)\mapsto \phi(y)(q(\gamma_\ell(y),u)-\psi(u))+\psi(u),$$
where $\gamma_\ell:[0,1]^d\rightarrow V,y\mapsto\frac{y-h}{\ell}+x_0$ and $h:=(\frac{1}{2},\dots,\frac{1}{2})$.

{\em Step 2:}
For $s\in\{1,\dots,m\}$ we have
\begin{eqnarray*}
 &&\sup\{\vert\partial_y^\beta q_\ell(y,u)\vert:y\in[0,1]^d,\beta\in\mathbb N^d,\vert\beta\vert = s\} \\ 
&\leq&\sup\{\vert q(\gamma_\ell(y),u)-\psi(u)\vert \vert\partial_y^\beta\phi(y)\vert :y\in[0,1]^d,\beta\in\mathbb N^d,\vert\beta\vert = s\}\\
 && + 2^mC_\phi\sup\{\vert \partial_y^\beta(q(\gamma_\ell(y),u)-\psi(u))\vert :y\in[0,1]^d,\beta\in\mathbb N^d,1\leq \vert\beta\vert \leq s\}\\
 &\leq& LC_\phi\sup\{\vert \gamma_\ell(y)-x_0\vert:y\in[0,1]^d\}\vert\Re(\psi(u))\vert\\
 && + 2^mC_\phi\frac{1}{\ell}\sup\{\vert \partial_x^\beta(q(x,u)-\psi(u))\vert :x\in V,\beta\in\mathbb N^d,1\leq \vert\beta\vert \leq s\}\\
&\leq&K_1\frac{\vert\Re(\psi(u))\vert}{\ell}
\end{eqnarray*}
with $K_1 := (1+2^m)LC_\phi$. By \cite[Theorem 3.2.16]{grafakos.08} there is another constant $K_2<\infty$ such that
$$\Vert q_\ell(\cdot,u)\Vert_{A(T)} \leq K_2\frac{\vert\Re(\psi(u))\vert}{\ell}$$
where 
$\Vert q_\ell(\cdot,u)\Vert_{A(T)}=\sum_{n\in\mathbb Z^d\setminus\{0\}}|c_n(u)|$ 
is the absolute sum of the Fourier coefficients 
$$c_n(u):=\int_{[0,1]^d}q_\ell(x,u)e^{-2\pi inx}dx, \quad n\in\mathbb Z^d$$ 
except for the coefficient $a_0(u):=c_0(u)$, 
which appears to be missing in the statement of \cite[Theorem 3.2.16]{grafakos.08}. 
Thus there is $\ell\geq 2L$ such that
$$\Vert q_\ell(\cdot,u)\Vert_{A(T)} \leq \frac{\vert\Re(\psi(u))\vert}{4}.$$
We also have
\begin{eqnarray*}
\vert \Re(a_0(u))-\Re(\psi(u))\vert&\leq&\vert a_0(u)-\psi(u)\vert\\
&\leq& \int_{[0,1]^d}\vert q_\ell(y,u)-\psi(u)\vert dy \\
&\leq& L\vert\Re(\psi(u))\vert\int_{[0,1]^d} \vert \gamma_\ell(y)-x_0\vert dy \\
&\leq& L\frac{\vert\Re(\psi(u))\vert}{\ell}\\
&\leq& \frac{\vert\Re(\psi(u))\vert}{2},
\end{eqnarray*}
which implies $\vert\Re(\psi(u))\vert\leq 2\vert\Re(a_0(u))\vert$. 
For $n\in\mathbb Z^d\backslash\{0\}$ set 
\begin{eqnarray*}
 a_n & := & e^{2\pi i n(h-\ell x_0)}c_n,\\
 b_n & := & ie^{2\pi i n(h-\ell x_0)}c_n.
\end{eqnarray*}
Then
\begin{eqnarray*}
\sum_{n\in\mathbb Z^d\backslash\{0\}}\left(\vert a_n(u)\vert+\vert b_n(u)\vert\right) &\leq &2\sum_{n\in\mathbb Z^d\backslash\{0\}}\vert c_n(u)\vert\\
&=& 2\Vert q_\ell(\cdot,u)\Vert_{A(T)}\\
&\leq &\frac{\vert\Re(\psi(u))\vert}{ and, h2}\\
&\leq &\vert\Re(a_0(u))\vert\\
&= &-\Re(a_0(u)),
\end{eqnarray*}
which implies (\ref{e:fc1}).

{\em Step 3:}
(\ref{e:fc2}) can be deduced by applying \cite[Proposition 3.1.2(10)]{grafakos.08} and the same arguments as in Step 2 to $\Delta_1q$ and using (\ref{e:20b}) instead of (\ref{e:20a}).

{\em Step 4:}
Let $U$ be the cube centered at $x_0$ with radius $1/(4\ell)$. For $x\in U$ with $y:=\gamma_\ell^{-1}(x)$ we have $y\in[1/4,3/4]^d$ and
$$q_\ell(y,u) = \phi(y)(q(\gamma_\ell(y),u)-\psi(u))+\psi(u) = q(x,u).$$
Define $\widetilde q_{x_0,\ell}(x,u):=\widetilde q_\ell(\gamma_\ell^{-1}(x),u)$, $x\in\mathbb R^d$,
where $\widetilde q_\ell(\cdot,u)$ denotes the periodic continuation of
$q_\ell(\cdot,u)$ to $\mathbb R^d$.
The inversion formula \cite[Proposition 3.1.14]{grafakos.08} yields
\begin{eqnarray*}
 \widetilde q_{x_0,\ell}(x,u) & = & \widetilde q_\ell(y,u)\\
        & = & \sum_{n\in\mathbb Z^d} c_n(u)e^{2\pi iny}\\
        & = & \sum_{n\in\mathbb Z^d} c_n(u)e^{2\pi i n(h-\ell x_0)}e^{2\pi i\ell nx}\\
        & = & \sum_{n\in\mathbb Z^d} \left(a_n(u)\cos\left({2\pi}{\ell}nx\right)+b_n(u)\sin\left({2\pi}{\ell}nx\right)\right)\\
        & = & \sum_{n\in\mathbb Z^d} \left(a_n(u)\cos(knx)+b_n(u)\sin(knx)\right)
\end{eqnarray*}
where $k:={2\pi}{\ell}$. Thus (\ref{e:sincos}) holds for $\widetilde q_{x_0,\ell}$. 

Moreover, $\widetilde q_\ell(\cdot,u)$ satisfies the requirements of Theorem \ref{Satz: Existenz unter Stetigkeit} and,
together with Steps 2, 3, those of Lemma \ref{Satz: Eindeutigkeit unter Fourier Bedingung}.
The localisation theorem \ref{S:Lokalisation} yields that existence and uniqueness holds for the $q$-martingale problem.
\end{proof}

\begin{proof}[Proof of Theorem \ref{Satz: Eindeutigkeit unter Elliptizitaet neu}]
 $q$ satisfies the requirements in Lemma \ref{Satz: Eindeutigkeit unter Elliptizitaet}.
\end{proof}

\appendix
\section{Convolutions and total variation}\label{s:appendix}
In this appendix we recall various properties of the total variation and the convolution of complex measures on $\mathcal B(\mathbb R^d)$.
A {\em complex measure} on $\mathbb R^d$ is a function $\mu:\mathcal B(\mathbb R^d)\rightarrow\mathbb C$ such that $\mu(\cup_{A\in \mathcal Z}A) = \sum_{A\in \mathcal Z}\mu(A)$ for any countable family $\mathcal Z\subset\mathcal B(\mathbb R^d)$ of pairwise disjoint sets, cf.\ \cite[Section \S3.4]{dinculeanu.67}.
We denote the set of complex measures on $\mathbb R^d$ by $\mathcal C(\mathbb R^d)$. A {\em decomposition of a measurable set $A$} is a finite system $\mathcal Z$ of pairwise disjoint measurable sets such that $\cup_{B\in \mathcal Z}B =A$. The {\em total variation measure} of $\mu$ is the measure defined by
 $$\vert\mu\vert (A) :=\sup\left\{ \sum_{B\in\mathcal Z} \vert Q(B)\vert : \mathcal Z\text{ is a decomposition of }A\right\},\quad A\in\mathcal B(\mathbb R^d).$$
 The {\em total variation of $\mu$} is defined by
 $ \Vert\mu\Vert := \vert\mu\vert(\mathbb R^d)$.
The {\em product measure} 
of complex measures $\mu,\nu$ on $\mathbb R^d$ resp.\ $\mathbb R^n$
is the complex measure $\mu\otimes\nu$ on $\mathbb R^d\times\mathbb R^n$
given by
 $$(\mu\otimes\nu)(A\times B) = \mu(A)\nu(B),\quad A\in\mathcal B(\mathbb R^d),B\in\mathcal B(\mathbb R^n).$$
 The {\em convolution} of complex measures $\mu,\nu$ on $\mathbb R^d$
 is the complex measure $\mu*\nu$ on $\mathbb R^d$ defined by
 $$(\mu*\nu)(A) = \int\mu(A-x)\nu(dx),\quad A\in\mathcal B(\mathbb R^d).$$
 Complex measures $\mu,\nu$ on $\mathbb R^d$ are called {\em orthogonal} if there is $A\in\mathcal B(\mathbb R^d)$ such that $\mu(B)=0$ for any Borel set $B\subset A$ and $\nu(C)=0$ for any Borel set $C\subset\mathbb R^d\backslash A$.
 The  {\em Dirac measure} concentrated in $a\in\mathbb R^d$ is denoted by $\delta_a$.
The {\em Fourier transform} of a complex measure $\mu$ on $\mathbb R^d$  
is the function $\hat\mu:\mathbb R^d\to\mathbb C$ given by
 $$\hat\mu(u) := \int_{}e^{i\<u,x\>}\mu(dx).$$
 A complex measure $\mu$ on $\mathbb R^d$ is {\em symmetric}
 (resp.\ {\em anti-symmetric}) if $\mu(A) = \mu(-A)$ 
 (resp.\ $\mu(A) = -\mu(-A)$) for any $A\in\mathcal B(\mathbb R^d)$.
 
 Let us recall several properties of complex measures, which can be found or easily derived from results
 in \cite{dinculeanu.67}.
\begin{lem}\label{p:Eigenschaften komplexer masse}
 Let $\mu$, $\nu$ be complex measures on $\mathbb R^d$ and $\eta$ a complex measure on $\mathbb R^n$. Then the following statements hold.
 \begin{enumerate}
  \item $\vert\mu\vert$ is an $\mathbb R_+$-valued (and hence finite) measure.
  \item $\vert\mu(A)\vert\leq\vert\mu\vert(A)$ for any $A\in\mathcal B(\mathbb R^d)$.
  \item $\mu$ is a regular measure in the sense of \cite[Definition \S15.2.1]{dinculeanu.67}.
  \item $\Vert\cdot\Vert$ is a complete norm on $\mathcal C(\mathbb R^d)$.
  \item (Hahn-Jordan decomposition) 
  There are $\mathbb R_+$-valued measures $\mu_1,\mu_2,\mu_3,\mu_4$ such that
  $$\mu = (\mu_1-\mu_2) + i(\mu_3-\mu_4),$$
  where $\mu_1$, $\mu_2$ are orthogonal and $\mu_3$, $\mu_4$ are orthogonal.
  \item A measurable function $f:\mathbb R^d\rightarrow\mathbb C$ is $\mu$-integrable if and only if it is $\vert\mu\vert$-integrable and in that case
  $$ \left\vert\int_{} fd\mu\right\vert\leq \int_{}\vert f\vert d\vert\mu\vert. $$
  \item Any bounded measurable function $f:\mathbb R^d\rightarrow\mathbb C$ is $\mu$-integrable.
 
  \item For any $\mu*\nu$-integrable function $f$ we have
    $$\int_{} f(v) (\mu*\nu)(dv) = \int_{}\int_{} f(v+w) \mu(v)\nu(w).$$
  \item $\vert\mu*\nu\vert$ is absolutely continuous with respect to $\vert\mu\vert*\vert\nu\vert$ with density bounded by $1$.
  \item $\Vert\mu*\nu\Vert\leq \Vert\mu\Vert\Vert\nu\Vert$
 \item $\vert\mu\otimes\eta\vert=\vert\mu\vert\otimes\vert\eta\vert$
  \item If $\mu$ and $\nu$ are orthogonal, then $\vert\mu+\nu\vert=\vert\mu\vert+\vert\nu\vert$.
  \item $(\mathcal C(\mathbb R^d),+,*,\Vert\cdot\Vert)$ is a complex commutative Banach algebra with unit $\delta_0$.
\item The Fourier transform $\hat{\ }$ is a one-to-one homomorphism of algebras which is continuous with respect to the total variation norm and the uniform norm, respectively. 
\item $\mu$ is symmetric resp.\ anti-symmetric if its Fourier transform is a symmetric resp.\ anti-symmetric function.
  \end{enumerate}
\end{lem}

For an introduction to analytic functional calculus see e.g.\ \cite[Definition 3.3.1]{palmer.94}. For a complex measure $\mu$ and a function $f:\mathbb C\to\mathbb C$ which is holomorphic on a neighbourhood of the spectrum of $\mu$ we write $f(\mu)$ for the complex measure obtained by the analytic functional calculus applied to $\mu$ and $f$.
\begin{lem}\label{l:symmetrie der variation}
 Let $\mu$ be a complex measure on $\mathbb R^d$. If $\mu$ is symmetric or anti-symmetric, then $\vert\mu\vert$ is symmetric. If $f:\mathbb C\to\mathbb C$ is holomorphic on a neighbourhood of the spectrum of $\mu$ and $\mu$ is symmetric, then $f(\mu)$ is symmetric. If $f$ is an odd entire function and $\mu$ is anti-symmetric, then $f(\mu)$ is anti-symmetric. If $f$ is an even entire function and $\mu$ is anti-symmetric, then $f(\mu)$ is symmetric.
\end{lem}
\begin{proof}
 Let $\mu$ be symmetric or anti-symmetric and $A\in\mathcal B(\mathbb R^d)$. Let $\epsilon>0$ and $\mathcal Z$ be a decomposition of $A$ such that $$\vert\mu\vert(A)\leq \epsilon+\sum_{B\in\mathcal Z}\vert\mu(B)\vert.$$
 Then
$$ \vert \mu\vert(A)-\epsilon\leq \sum_{B\in\mathcal Z}\vert\mu(-B)\vert 
                            \leq \vert\mu\vert(-A).$$
 Hence $\vert\mu\vert(A)\leq \vert\mu\vert(-A)\leq \vert\mu\vert(A)$, which implies symmetry.

 Let $\mu$ be symmetric, $f$ an entire function and $z\in\mathbb C$. Then $z\delta_0-\mu$ is symmetric. If $z$ is not in the spectrum of $\mu$, then the measure $R(z,\mu)$ with the property $(z\delta_0-\mu)*R(z,\mu)=\delta_0$ is symmetric as well. Thus
  $$ f(\mu) = \frac{1}{2\pi i}\int_{\Gamma} f(z)R(z,\mu)dz$$
 is symmetric as well, where $\Gamma$ is a suitable integration path.
 
 Now let $\mu$ be anti-symmetric and $z\in\mathbb C$ outside of the spectrum of $\mu$. Then
  \begin{align*}
   (R(z,\mu)-R(-z,\mu))\hat{\ }(u) = \frac{1}{z+\hat\mu(u)}+\frac{1}{z-\hat\mu(u)},\quad u\in\mathbb R^d.
  \end{align*}
 By Lemma \ref{p:Eigenschaften komplexer masse}(15) $\alpha_z:=R(z,\mu)-R(-z,\mu)$ is symmetric. Similar arguments yield that $\beta_z:=R(z,\mu)+R(-z,\mu)$ is anti-symmetric and we obviously have $R(z,\mu)=1/2(\alpha_z+\beta_z)$. Let $\Gamma$ be a symmetric path around the spectrum of $\mu$. Then
 \begin{equation}\label{e:fformel} 
  f(\mu) = \frac{1}{4\pi i}\int_{\Gamma}f(z)\alpha_zdz + \frac{1}{4\pi i}\int_{\Gamma}f(z)\beta_zdz.
 \end{equation}
  Observe that the first summand is symmetric and the second summand is anti-symmetric. If $f$ is even, then the first summand vanishes and hence $f(\mu)$ is symmetric. If $f$ is odd, then the second summand vanishes and hence $f(\mu)$ is anti-symmetric.
\end{proof}

\begin{lem}\label{l:spezielle Symmetrie}
 Let $\mu=z(\delta_a-\delta_{-a})$ for some $a\in\mathbb R^d$, $z\in\mathbb C$. Then $\vert\exp(\mu)\vert$ is symmetric.
\end{lem}
\begin{proof}
 For $w\in\mathbb C$ with $\vert w\vert>2|z|$ define
 \begin{align*}
  \alpha_w:=&\ R(w,\mu)-R(-w,\mu), \\
   \beta_w:=&\ R(w,\mu)+R(-w,\mu),
 \end{align*}
 where $R(w,\mu)$ is the complex measure such that $R(w,\mu)*(w\delta_0-\mu)=\delta_0$. 
From their Fourier transforms we conclude that that $\alpha_w$ is supported on $O:=\{ka:k\in\mathbb Z,k\text{ is odd}\}$ and $\beta_w$ is supported on $E:=\{ka:k\in\mathbb Z,k\text{ is even}\}$. 
 By (\ref{e:fformel}) and the subsequent observation, $\sinh(\mu)$ is concentrated on $O$ while $\cosh(\mu)$ is concentrated on $E$. Proposition \ref{p:Eigenschaften komplexer masse} yields
 $$\vert\exp(\mu)\vert=\vert\sinh(\mu)\vert+\vert\cosh(\mu)\vert.$$
 Lemma \ref{l:symmetrie der variation} yields that $\vert\exp(\mu)\vert$ is the sum of symmetric measures and hence symmetric.
\end{proof}

\begin{rem}\label{b:Momente}
 Let $\mu$ be a complex measure on $\mathbb R^d$ such that $\int\vert v\vert^2|\mu|(dv)<\infty$. Then
\begin{align*}
 \exp(\mu)(\mathbb R^d) =&\ \exp(\mu(\mathbb R^d)), \\
 \int_{} v\exp(\mu)(dv) =&\ \int_{}v\mu(dv)\exp(\mu(\mathbb R^d)), \\
 \int_{} \vert v\vert^2\exp(\mu)(dv) =&\ \left(\int_{}\vert v\vert^2\mu(dv)+2\left\vert\int_{}v\mu(dv)\right\vert^2\right)\exp(\mu(\mathbb R^d)).
\end{align*}
\end{rem}

\section{Infinite product measures and convolutions}
In this section we recall the definition and properties for infinite product measures and infinite convolution 
for complex Borel measures on $\mathbb R^d$.
Denote by $\mathcal B(\mathbb R^d)^\mathbb N$  the $\sigma$-algebra on $(\mathbb R^d)^{\mathbb N}$ which is generated 
by the mappings
 $$ \pi_n :(\mathbb R^d)^{\mathbb N} \rightarrow \mathbb R^d,\quad  (x_k)_{k\in\mathbb N} \mapsto x_n. $$
It is also generated by the algebra
 $$ \mathcal R:=\bigcup_{n\in\mathbb N} \sigma(\pi_1,\dots,\pi_n).$$

\begin{defn}
  Let $(\mu_n)_{n\in\mathbb N}$ be a sequence of complex Borel measures on $\mathbb R^d$
  and define  $a_n:=\mu_n(\mathbb R^d)$, $n\in\mathbb N$. Assume that $a:=\Pi_{n=1}^\infty a_n$ exists in $\mathbb C$.
  A complex Borel measure $\mu$ on $((\mathbb R^d)^\mathbb N,\mathcal B(\mathbb R^d)^\mathbb N)$ is the {\em infinite product measure} of $(\mu_n)_{n\in\mathbb N}$ if 
 \begin{equation}\label{e:infprod}
  \int f(\pi_1,\dots,\pi_n) d\mu = \frac{a}{\Pi_{j=1}^na_j}\int\cdots\int f(x_1,\dots,x_n) \mu_1(dx_1),\dots,\mu_n(dx_n) 
 \end{equation}
      for any $n\in\mathbb N$ and any bounded measurable function $f:(\mathbb R^d)^n\rightarrow \mathbb R$.
\end{defn} 
If it exists, the infinite product measure is unique because it is unique on $\mathcal R$.
It is denoted by $\otimes_{n=1}^\infty\mu_n$.
The following existence statement is classical.
\begin{prop}\label{p:unend product}
   Let $(P_n)_{n\in\mathbb N}$ be a sequence of probability measures on $\mathbb R^d$. Then there is a unique probability measure $P$ on $((\mathbb R^d)^{\mathbb N},\mathcal B(\mathbb R^d)^\mathbb N)$ such that $(\pi_n)_{n\in\mathbb N}$ is a sequence of independent random variables with 
   $P^{\pi_n} = P_n$.
\end{prop}

This can be easily lifted to finite measures as long as the product of their total mass converges to a finite non-zero number.

\begin{lem}\label{l:unend product endl}
  Let $(\mu_n)_{n\in\mathbb N}$ be a sequence of finite Borel measures on $\mathbb R^d$ and define $a_n:=\mu_n(\mathbb R^d)$ for any $n\in\mathbb N$. Assume that $a:=\Pi_{n=1}^\infty a_n\in(0,\infty)$. Then the infinite product measure of $(\mu_n)_{n\in\mathbb N}$ exists.
\end{lem}
\begin{proof}
  Define $P_n := \mu_n/a_n$. Then $(P_n)_{n\in\mathbb N}$ satisfies the requirements of Proposition \ref{p:unend product} and, hence, there is a probability measure $P$ as in Proposition \ref{p:unend product}. The measure $\mu:=aP$ has the required property.
\end{proof}

\begin{prop}\label{p:unend product compl}
  Let $(\mu_n)_{n\in\mathbb N}$ be a sequence of complex Borel measures on $\mathbb R^d$ and define $a_n:=\mu_n(\mathbb R^d)$ and $c_n:=|\mu_n|(\mathbb R^d)$ for any $n\in\mathbb N$. Assume that $c:=\Pi_{n=1}^\infty c_n\in(0,\infty)$ and that $a:=\Pi_{n=1}^\infty a_n$ exists in $\mathbb C$. 
  Then the infinite product measure 
  $(\mu_n)_{n\in\mathbb N}$ exists and we have 
$|\otimes_{n=1}^\infty\mu_n|=\otimes_{n=1}^\infty|\mu_n|$.
\end{prop}
\begin{proof}
By Lemma \ref{l:unend product endl} the infinite product measure $\nu$ of $(|\mu_n|)_{n\in\mathbb N}$ exists.
Define a mapping $\tilde\mu:\mathcal R\to\mathbb C$ via 
\begin{equation}\label{e:mutilde}
 \tilde\mu((\pi_1,\dots,\pi_n)^{-1}(A)) = \frac{a}{\Pi_{j=1}^na_j}\int\cdots\int 1_A(x_1,\dots,x_n) \mu_1(dx_1),\dots,\mu_n(dx_n)
\end{equation}
 for $A\in\mathcal B(\mathbb R^d)^n, n\in\mathbb N$.
It is easy to verify that $\tilde\mu$ is a well-defined finitely additive measure on $\mathcal R$.
Since $\nu$ is $\sigma$-additive and hence continuous in $\varnothing$, this also holds for 
$\tilde\mu$.
Carath\'eodory's extension theorem yields that $\tilde\mu$ can be extended to a measure $\mu$ on $\mathcal B(\mathbb R^d)^\mathbb N$.
Equation (\ref{e:mutilde}) implies that (\ref{e:infprod}) holds and hence $\mu$ is the product measure of 
$(\mu_n)_{n\in\mathbb N}$.
From Lemma \ref{p:Eigenschaften komplexer masse}(11) it follows that $|\mu|$ coincides with 
$\otimes_{n=1}^\infty |\mu_n|$ on $\mathcal R$ and hence on $\mathcal B(\mathbb R^d)^\mathbb N=\sigma(\mathcal R)$.
\end{proof}

Now we turn to infinite convolutions which roughly coincide with pushforward measures of infinite  sums of independent random variables. 
\begin{defn}
 Let $(\mu_n)_{n\in\mathbb N}$ be a sequence of complex Borel measures on $\mathbb R^d$. A complex Borel measure $\eta$  on $\mathbb R^d$ is the {\em infinite convolution} of $(\mu_n)_{n\in\mathbb N}$ if
  $$ \int f(x) \eta(dx) = \lim_{n\rightarrow\infty}\int f(x)(*_{k=1}^n\mu_k)(dx)$$
  for any bounded continuous $f:\mathbb R^d\rightarrow \mathbb R$.
\end{defn}

Observe that the infinite convolution is uniquely defined by the limiting property if it exists.
It is denoted as $*_{n=1}^\infty\mu_n$.
We give a simple criterion for its existence.
\begin{prop}\label{p:infinite convolution}
   Let $(\mu_n)_{n\in\mathbb N}$ be a sequence of complex Borel measures on $\mathbb R^d$ and define $c_n:=|\mu_n|(\mathbb R^d)$ and $a_n:=\mu_n(\mathbb R^d)$ for any $n\in\mathbb N$. Assume that $c:=\Pi_{n=1}^\infty c_n\in(0,\infty)$, that $a:=\Pi_{n=1}^\infty a_n\in\mathbb C$, that
  $\sum_{n=1}^\infty \int |v|^2 |\mu_n|(dv) <\infty$,
  and that $\int v |\mu_n|(dv) =0$, $n\in\mathbb N$.  
Then the infinite convolutions of $(\mu_n)_{n\in\mathbb N}$ as well as $(|\mu_n|)_{n\in\mathbb N}$ exist.
Moreover,
$|*_{n=1}^\infty\mu_n|\leq*_{n=1}^\infty|\mu_n|$.
\end{prop}
\begin{proof}
The measure $P:=|\otimes_{n=1}^\infty\mu_n|/c$ is a probability measure 
and $\pi_n$ is a sequence of independent random variables relative to $P$.
Since $\lim_{n\to\infty}c_n=1$, we have
$$\sum_{n=1}^\infty \int |\pi_n|^2 dP =\sum_{n=1}^\infty c_n^{-1}\int |v|^2 |\mu_n|(dv) <\infty$$ 
and hence $(\sum_{k=1}^n\pi_k)_{n\in\mathbb N}$
 converges in  $L^2((\mathbb R^d)^{\mathbb N},\mathcal B(\mathbb R^d)^\mathbb N,P)$ 
 and hence in probability to some random variable $S$.
Set $\eta:=(\otimes_{n=1}^\infty\mu_n)^S$. For any bounded continuous function $f:\mathbb R^d\to\mathbb R$ we obtain
  \begin{align*}
   \left|\int \bigg(f(S)-f\bigg(\sum_{k=1}^n\pi_k\bigg)\bigg) d(\otimes_{n=1}^\infty\mu_n)\right| 
                       &\leq \int \bigg|f(S)-f\bigg(\sum_{k=1}^n\pi_k\bigg)\bigg| d|\otimes_{n=1}^\infty\mu_n| \\
   &=c\int \bigg|f(S)-f\bigg(\sum_{k=1}^n\pi_k\bigg)\bigg| dP \\
      &\to0\\
   \end{align*}
 for $n\to\infty$ and hence 
 \begin{align*}
   \int f(x) \eta(dx) &= \int f(S) d(\otimes_{n=1}^\infty\mu_n) \\
                       &= \lim_{n\rightarrow \infty} \int f\bigg(\sum_{k=1}^n\pi_k\bigg) d(\otimes_{n=1}^\infty\mu_n) \\    
                       &= \lim_{n\rightarrow \infty} {a\over\prod_{k=1}^na_k}\int\cdots\int f\bigg(\sum_{k=1}^nx_k\bigg) \mu_1(dx_1)\cdots\mu_n(dx_n) \\
                        &= \lim_{n\rightarrow\infty}\int f(x)(*_{k=1}^n\mu_k)(dx)
 \end{align*}
as desired for the infinite convolution.
The existence of $*_{n=1}^\infty|\mu_n|$ follows along the same lines by setting
$\gamma:=|\otimes_{n=1}^\infty\mu_n|^S$.

The last statement follows from $|(\otimes_{n=1}^\infty\mu_n)^S|\leq|\otimes_{n=1}^\infty\mu_n|^S$.
 \end{proof}


\end{document}